\documentclass[leqno, nowrap]{amsart}

\usepackage{amsmath}
\numberwithin{equation}{section}

\usepackage{amsfonts}
\usepackage{amssymb}
\usepackage[top=1.5in, bottom=1.5in, left=1.5in, right=1.5in]{geometry}

\usepackage{amsthm}
\newtheorem{Theorem}{Theorem}[section]
\newtheorem{Lemma}[Theorem]{Lemma}
\newtheorem{Corollary}[Theorem]{Corollary}
\newtheorem{Proposition}[Theorem]{Proposition}

\usepackage{indentfirst}

\usepackage{endnotes}

\newcommand{\cal}{\mathcal}
\newcommand{\R}{\mathbb{R}}
\newcommand{\del}{\partial}

\newcommand{\beq}{\begin{equation}}
\newcommand{\eeq}{\end{equation}}
\newcommand{\bml}{\begin{multline}}
\newcommand{\eml}{\end{multline}}
\newcommand{\bTh}{\begin{Theorem}}
\newcommand{\eTh}{\end{Theorem}}
\newcommand{\bCo}{\begin{Corollary}}
\newcommand{\eCo}{\end{Corollary}}
\newcommand{\bLem}{\begin{Lemma}}
\newcommand{\eLem}{\end{Lemma}}
\newcommand{\bProp}{\begin{Proposition}}
\newcommand{\eProp}{\end{Proposition}}

\newcommand{\LM}[1]{{L^{#1}(M)}}

\newcommand{\Lp}[2]{{L^{#1}(#2)}}

\newcommand{\leC}{\lesssim}
\newcommand{\Ntube}[1]{\cal{N}_{#1} (\gamma)}
\newcommand{\eps}{\varepsilon}
\renewcommand{\phi}{\varphi}

\begin{document}

\hyphenpenalty=1000000

\title{Concentration of eigenfunctions near a concave boundary}
\author{Sinan Ariturk}
\begin{abstract}
This paper concerns the concentration of Dirichlet eigenfunctions of the Laplacian on a compact two-dimensional Riemannian manifold with strictly geodesically concave boundary.
We link three inequalities which bound the concentration in different ways.
We also prove one of these inequalities, which bounds the $L^p$ norms of the restrictions of eigenfunctions to broken geodesics.
\end{abstract}

\maketitle

\section{Introduction}

	Let $(M,g)$ be a compact two-dimensional Riemannian manifold with smooth boundary.
	Assume that the boundary is strictly geodesically concave.
	This means that for any point $x$ in $\del M$, there is a geodesic in $M$ which goes through $x$ intersecting $\del M$ tangentially with exactly first order contact.
	Let $e_j$ be Dirichlet eigenfunctions of the Laplacian $\Delta_g$ which form an orthonormal basis of $L^2(M)$.
	Let $0 < \lambda_0 \le \lambda_1 \le \lambda_2 \le \ldots$ be the corresponding eigenvalues, normalized so that $-\Delta_g e_j = \lambda_j^2 e_j$.
	This paper concerns the concentration of the eigenfunctions $e_j$.
	
	One way to measure the concentration of the eigenfunctions is by their $L^p$ norms.
	For $p \ge 2$, the eigenfunctions satisfy
	\beq
	\label{sogest}
		\| e_j \|_\LM{p} \leC \lambda_j^{\delta(p)}
	\eeq
	where
	\[
		\delta(p) = 
		\begin{cases} 
			\frac{1}{4}-\frac{1}{2p} & \text{if } 2 \le p \le 6 \\
			\frac{1}{2}-\frac{2}{p} & \text{if } 6 \le p \le \infty
		\end{cases}
	\]
	This was proven by Grieser \cite{G}.
	We can interpret \eqref{sogest} as a way of bounding the concentration of the eigenfunctions.
	For $p>2$, a natural problem is to determine when \eqref{sogest} is sharp, meaning
	\beq
	\label{sogsharp}
		\limsup_{j \to \infty} \lambda_j^{-\delta(p)} \| e_j \|_\LM{p} > 0
	\eeq
	We will give two conditions which are equivalent to \eqref{sogsharp} when $2<p<6$.
	Specifically, we will consider two other inequalities which measure the concentration of eigenfunctions.
	We will then see that sharpness of these inequalities is equivalent to \eqref{sogsharp} when $2<p<6$.
	
	Our second way of measuring the concentration of eigenfunctions is by the $L^p$ norms of their restrictions to broken geodesics.
	A broken geodesic is a curve in $M$ which is geodesic away from the boundary and reflects off the boundary according to the reflection law for $g$.
	We bound this kind of concentration in the following theorem.
	
	\bTh
	\label{thbgt}
		If $\gamma$ is a broken geodesic of unit length in $M$, then
		\[
			\| e_j \|_\Lp{p}{\gamma} \leC \lambda_j^{\sigma(p)}
		\]
		where
		\[
		\sigma(p) = 
		\begin{cases} 
			\frac{1}{4} & \text{if } 2 \le p \le 4 \\
			\frac{1}{2}-\frac{1}{p} & \text{if } 4 \le p \le \infty
		\end{cases}
		\]
	\eTh
	
	This extends a result of Burq-G\'erard-Tzvetkov \cite{BGT}.
	Their result dealt with compact two-dimensional Riemannian manifolds without boundary.
	Their work was motivated by Reznikov \cite{R} who considered hyperbolic surfaces.
	Note that in proving Theorem \ref{thbgt}, it suffices to prove the case $p=4$.
	The case $p=\infty$ follows from \eqref{sogest} since the eigenfunctions are continuous.
	Then interpolation will yield the cases $4<p<\infty$, and H\"older's inequality will yield the cases $2\le p < 4$.
	Another way to bound the $L^2$ norms over broken geodesics is given by the following corollary.
	
	\bCo
	\label{cob}
		If $\gamma$ is a broken geodesic of unit length, $p \ge 2$, and $\eps>0$, then there is a constant $C_\eps$ such that
		\[
			\| e_j \|_\Lp{2}{\gamma} \le C_\eps \lambda_j^\frac{1}{2p} \| e_j \|_\LM{p} + \eps \lambda_j^\frac{1}{4}
		\]
	\eCo
	
	For two-dimensional manifolds without boundary, Bourgain \cite{B} gave a stronger version of this inequality, without the second term in the right side.
	In section 5, we will use his result and Theorem \ref{thbgt} to prove Corollary~\ref{cob}.
	
	We will link sharpness of Theorem \ref{thbgt} for $p=2$ and sharpness of \eqref{sogest} for $2<p<6$.
	Let $\Pi$ be the set of all unit length broken geodesics in $M$.
	We will show that for $2<p<6$, the inequality \eqref{sogsharp} is equivalent to
	\[
	\label{bgtsharp}
		\limsup_{j \to \infty} \sup_{\gamma \in \Pi} \lambda_j^{-\frac{1}{4}} \| e_j \|_\Lp{2}{\gamma} > 0
	\]
	
	Our third way of measuring the concentration of eigenfunctions is by their $L^2$ norms over neighborhoods of broken geodesics.
	For $\gamma$ in $\Pi$, define the neighborhoods
	\[
		\cal{N}_j(\gamma) = \Big\{ x \in M : d_g(x, \gamma) < \lambda_j^{-\frac{1}{2}} \Big\}
	\]
	Here $d_g$ is the Riemannian distance function corresponding to $g$.
	Trivially, we have
	\[
		\| e_j \|_\Lp{2}{\Ntube{j}} \le 1
	\]
	For $2<p<6$, we will also show that \eqref{sogsharp} is equivalent to
	\[
		\limsup_{j \to \infty} \sup_{\gamma \in \Pi} \| e_j \|_\Lp{2}{\Ntube{j}} > 0
	\]
	This will be a consequence of the following theorem.
	
	\bTh
	\label{thlink}
		Assume $\Lambda$ is large and fix $\eps>0$.
		There is a constant $C_\eps$ such that for $\lambda_j \ge \Lambda$, the eigenfunctions $e_j$ satisfy
		\[
			\| e_j \|_\LM{4}^4 \le C_\eps \lambda_j^\frac{1}{2} \sup_{\gamma \in \Pi} \| e_j \|_\Lp{2}{\cal{N}_j(\gamma)}^2 + \eps \lambda_j^\frac{1}{2} + C
		\]
	\eTh
	
	This extends a result of Sogge \cite{Sog}, who considered compact two-dimensional Riemannian manifolds without boundary.
	Corollary \ref{cob} and Theorem \ref{thlink} imply the following result.
	
	\bCo
	\label{clink}
		Let $e_{j_k}$ be a subsequence of eigenfunctions and let $2<p<6$.
		The following are equivalent:
		\beq
		\label{sogsh}
			\limsup_{k \to \infty} \lambda_{j_k}^{-\delta(p)} \| e_{j_k} \|_\LM{p} > 0
		\eeq
		\beq
		\label{tubesh}
			\limsup_{k \to \infty} \sup_{\gamma \in \Pi} \| e_{j_k} \|_\Lp{2}{\Ntube{j_k}} > 0
		\eeq
		\beq
		\label{bgtsh}
			\limsup_{k \to \infty} \sup_{\gamma \in \Pi} \lambda_{j_k}^{-\frac{1}{4}} \| e_{j_k} \|_\Lp{2}{\gamma} > 0
		\eeq
	\eCo
	
	If \eqref{sogsh} holds for some $p$ in the range $2<p<6$, then it holds for all such $p$, by \eqref{sogest} and interpolation.
	So to prove Corollary \ref{clink}, it suffices to consider the case $p=4$.
	In this case, \eqref{sogsh} implies \eqref{tubesh} by Theorem \ref{thlink}.
	It is clear that \eqref{tubesh} implies \eqref{bgtsh}, and \eqref{bgtsh} implies \eqref{sogsh} by Corollary \ref{cob}.
	
	A related problem is to determine when a subsequence $e_{j_k}$ of eigenfunctions is quantum ergodic.
	To define this condition, let $S^*M$ be the unit cosphere bundle.
	The eigenfunctions $e_j$ induce distributions $U_j$ on $S^*M$ defined by
	\[
		U_j(a)=\Big \langle Op(a)e_j, e_j \Big \rangle
	\]
	where $Op(a)$ is the pseudodifferential operator, for a fixed quantization, with complete symbol $a$.
	To say a subsequence $e_{j_k}$ of eigenfunctions is quantum ergodic means that the weak* limit of the distributions $U_{j_k}$ is the normalized Liouville measure on $S^*M$.
	This definition is independent of the choice of quantization.
	In particular, this implies that the probability measures $|e_{j_k}|^2 \,dx$ converge weakly to the normalized Riemannian measure.
	In this case \eqref{tubesh} cannot hold, so Corollary \ref{clink} implies the following.
	
	\bCo
	Assume a subsequence $e_{j_k}$ of eigenfunctions is quantum ergodic.
	Then
	\[
		\limsup_{k \to \infty} \sup_{\gamma \in \Pi} \lambda_{j_k}^{-\frac{1}{4}} \| e_{j_k} \|_\Lp{2}{\gamma} = 0
	\]
	and for $2<p<6$,
	\[
		\limsup_{k \to \infty} \lambda_{j_k}^{-\delta(p)} \| e_{j_k} \|_\LM{p} = 0
	\]	
	\eCo

	Zelditch-Zworski \cite{ZZ} proved that if the billiard flow is ergodic, then there is a subsequence $e_{j_k}$ of density one which is quantum ergodic.
	A subsequence is of density one when
	\[
		\lim_{k \to \infty} \frac{k}{j_k} = 1
	\]
	Their result demonstrates that the global dynamics of the billiard flow influence the concentration of eigenfunctions.
	Our last result also demonstrates this.
	
	\bProp
	\label{open}
		Fix a broken geodesic $\gamma$  in $M$ of unit length which is not contained in a periodic broken geodesic.
		Then
		\[
			\limsup_{j \to \infty} \lambda_j^{-\frac{1}{4}} \| e_j \|_\Lp{2}{\gamma} = 0
		\]
	\eProp
	
	That is, if Theorem 1.1 is sharp for a fixed broken geodesic, then it must be a segment of a periodic broken geodesic.
	
\subsection*{Acknowledgements}
	I would like to thank Christopher Sogge for suggesting this problem and for his invaluable guidance.

\section{Reductions}
	
	The beginning of the proofs of Theorem ~\ref{thbgt} and Theorem ~\ref{thlink} are similar so we begin both in this section.
	We can assume that $M$ is a subset of a boundaryless compact two-dimensional Riemannian manifold $(M_0, g)$.
	Let $d_0$ be the Riemannian distance function on $M_0$ corresponding to $g$ and let $\Delta_0$ be the Laplacian on $M_0$.
	For the rest of this paper, we will assume $\lambda \ge 1$.
	
	Fix a small $\delta > 0$, and choose a $\chi \in \mathcal{S}(\mathbb{R})$
		with $\chi(0)=1$ and $\hat{\chi}$ supported on a closed interval contained strictly inside of $(\frac{1}{2}\delta, \delta)$.
	Define the translations $\chi _\lambda (s)=\chi(s-\lambda)$.
	We will use the operators $\chi_\lambda(\sqrt{-\Delta_g})$ and $\chi_\lambda(\sqrt{-\Delta_0})$.
	Here $\sqrt{-\Delta_g}$ is defined with respect to Dirichlet boundary conditions.
	Notice
	\[
		\chi_\lambda(\lambda) = 1
	\]
	Define $\rho_\lambda(s)=\chi_\lambda(s)+\chi_\lambda(-s)$.
	For large $\lambda$, we have
	\[
		1/2  \le | \rho_\lambda(\lambda) |
	\]
	Define the set
	\begin{equation*}
		H_\delta = \Big\{x \in M : d_g(x, \partial M) \le \delta \Big\}
	\end{equation*}
	and let $E_\delta$ be the complement of $H_\delta$ in $M$.
	To prove Theorem ~\ref{thbgt}, it suffices to prove that
		\begin{equation}
		\label{specgeorestineq}
					\| \rho_\lambda(\sqrt{-\Delta_g}) f \|_{L^4(\gamma \cap E_\delta)} + \| \chi_\lambda(\sqrt{-\Delta_g}) f \|_{L^4(\gamma \cap H_\delta)} \lesssim \lambda^{1/4} \| f \|_{L^2(M)}
			\end{equation}
	We have the following analogue.
	
	\begin{Theorem}
		\label{bgtth}
			If $\gamma$ is a smooth curve on $M_0$ of unit length, then
			\[
					\| \rho_\lambda(\sqrt{-\Delta_0}) f \|_{L^4(\gamma)} + \| \chi_\lambda(\sqrt{-\Delta_0}) f \|_{L^4(\gamma)} \lesssim \lambda^{1/4} \| f \|_{L^2(M_0)}
				\]
		\end{Theorem}

	Burq-G\'erard-Tzvetkov proved this inequality for $\chi_\lambda$, and the inequality for $\rho_\lambda$ follows easily from the following lemma, which we will prove later.
	
	\begin{Lemma}
	\label{neglap}
		The kernel of $\chi_\lambda(-\sqrt{-\Delta_0})$ is uniformly bounded, independent of $\lambda$.
	\end{Lemma}
		
	Let $\Pi_0$ be the set of all unit length geodesics in $M_0$.
	Fix $r \in (0,1)$.
	For $\gamma \in \Pi_0$, define the neighborhoods
	\[
			\mathcal{T}_\lambda(\gamma) = \Big\{x \in M_0: d_0(x,\gamma) < r\lambda^{-1/2} \Big\}
		\]
	There is a constant $\Lambda$ such that for any geodesic $\gamma \in \Pi_0$, there exists a fixed finite number of broken geodesics $\gamma_i \in \Pi$
		such that $\mathcal{T}_{\lambda_j}(\gamma) \cap M \subset \bigcup \mathcal{N}_j(\gamma_i)$ for $\lambda_j \ge \Lambda$.
	By \eqref{sogest}, we know $\| e_j \|_{L^4(M)} \lesssim \lambda_j^{1/8}$, so to prove Theorem ~\ref{thlink} it suffices to show that
	\begin{multline}
	\label{propineq}
			\int_{E_\delta} | \rho _\lambda(\sqrt{-\Delta_g}) f(x)|^2 |g(x)|^2 \,dx + \int_{H_\delta} | \chi _\lambda(\sqrt{-\Delta_g}) f(x)|^2 |g(x)|^2 \,dx \le
			\\
			C_\varepsilon \lambda^{1/2} \| f \|_{L^2(M)}^2 \sup_{\gamma \in \Pi_0} \| g \|_{L^2(\mathcal{T}_{\lambda}(\gamma))}^2
			+~\varepsilon \lambda^{1/4} \| f \|_{L^2(M)}^2 \| g \|_{L^4(M)}^2
			+~C \| f \|_{L^2(M)}^2 \| g \|_{L^2(M)}^2			
	\end{multline}
	We have the following analogue.
	\begin{Theorem}
		Fix $\varepsilon>0$.
		There is a constant $C_\varepsilon$ such that
		\label{soggeth}
		\begin{multline*}
			\int_{M_0} | \rho _\lambda(\sqrt{-\Delta_0}) f(x)|^2 |g(x)|^2 \,dx \le + \int_{M_0} | \chi _\lambda(\sqrt{-\Delta_0}) f(x)|^2 |g(x)|^2 \,dx \le
			\\
			C_\varepsilon \lambda^{1/2} \| f \|_{L^2(M_0)}^2 \sup_{\gamma \in \Pi_0} \| g \|_{L^2(\mathcal{T}_{\lambda}(\gamma))}^2
			+~\varepsilon \lambda^{1/4} \| f \|_{L^2(M_0)}^2 \| g \|_{L^4(M_0)}^2
			+~C \| f \|_{L^2(M_0)}^2 \| g \|_{L^2(M_0)}^2			
		\end{multline*}			
	\end{Theorem}
	For $r=1$, Sogge ~\cite{Sog} proved this inequality for $\chi_\lambda$.
	Moreover, the same proof shows this holds for smaller values of $r$ as well, and the inequality for $\rho_\lambda$ follows easily from Lemma \ref{neglap}.
	
	Define projection operators $\Pi_j$ on $L^2(M)$ by $\Pi_j f=\langle f , e_j \rangle e_j$.
	For $f$ in $L^2(M)$,
	\begin{multline}
		\label{wave}
			\chi _\lambda(\sqrt{-\Delta_g}) f
			=\sum_{j=0}^\infty \chi_\lambda(\lambda_j)  \Pi_j f
			=(2\pi)^{-1} \int \hat{\chi}(t) e^{-it\lambda} \sum_{j=0}^\infty e^{it\lambda_j} \Pi_j f \, dt
			\\
			=(2\pi)^{-1} \int \hat{\chi}(t) e^{-it\lambda} e^{it\sqrt{-\Delta_g}}f \,dt
		\end{multline}
	Likewise,
	\[
		\chi _\lambda(-\sqrt{-\Delta_g}) f = (2\pi)^{-1} \int \hat{\chi}(t) e^{-it\lambda} e^{-it\sqrt{-\Delta_g}}f \,dt
	\]
	which yields
	\[
		\rho _\lambda(\sqrt{-\Delta_g}) f = \pi^{-1} \int \hat{\chi}(t) e^{-it\lambda} \cos(t\sqrt{-\Delta_g})f \,dt
	\]

	Similarly, for $f$ in $L^2(M_0)$,
	\begin{equation}
		\label{freewave}
			\chi _\lambda(\sqrt{-\Delta_0}) f=(2\pi)^{-1} \int \hat{\chi}(t) e^{-it\lambda} e^{it\sqrt{-\Delta_0}}f \,dt
		\end{equation}
	and
	\[
		\rho _\lambda(\sqrt{-\Delta_0}) f = \pi^{-1} \int \hat{\chi}(t) e^{-it\lambda} \cos(t\sqrt{-\Delta_0})f \,dt
	\]
	If $t$ is in supp $\hat{\chi}$, then
	\begin{equation*}
		\Big{(}\cos(t\sqrt{-\Delta_g})f \Big{)}\Big{\vert}_{E_\delta} = \Big{(}\cos(t\sqrt{-\Delta_0})f \Big{)}\Big{\vert}_{E_\delta}
	\end{equation*}
	which implies that
	\begin{equation}
	\label{tycos}
		\Big{(}\rho_\lambda(\sqrt{-\Delta_g})f \Big{)}\Big{\vert}_{E_\delta} = \Big{(}\rho_ \lambda(\sqrt{-\Delta_0})f \Big{)}\Big{\vert}_{E_\delta}
	\end{equation}
	For a broken geodesic $\gamma$ on $M$ of unit length, Theorem ~\ref{bgtth} yields
	\[
			\| \rho_\lambda(\sqrt{-\Delta_g}) f \|_{L^4(\gamma \cap E_\delta)} \lesssim \lambda^{1/4} \| f \|_{L^2(M)}
		\]		
	So to prove \eqref{specgeorestineq}, it remains to prove
	\begin{equation}
	\label{hdeltageo}
			\| \chi_\lambda(\sqrt{-\Delta_g}) f \|_{L^4(\gamma \cap H_\delta)} \lesssim \lambda^{1/4} \| f \|_{L^2(M)}
		\end{equation}
	Similarly, Theorem ~\ref{soggeth} yields
	\begin{multline*}
			\int_{E_\delta} | \rho _\lambda(\sqrt{-\Delta_g}) f(x)|^2 |g(x)|^2 \,dx \le
			C_\varepsilon \lambda^{1/2} \| f \|_{L^2(M)}^2 \sup_{\gamma \in \Pi_0} \| g \|_{L^2(\mathcal{T}_{\lambda}(\gamma))}^2
		\\
			+~\varepsilon \lambda^{1/4} \| f \|_{L^2(M)}^2 \| g \|_{L^4(M)}^2
			+~C \| f \|_{L^2(M)}^2 \| g \|_{L^2(M)}^2
		\end{multline*}	
	So to prove \eqref{propineq}, it remains to prove
	\begin{multline}
		\label{hdelta}
			\int_{H_\delta} | \chi _\lambda(\sqrt{-\Delta_g}) f(x)|^2 |g(x)|^2 \,dx \le
			C_\varepsilon \lambda^{1/2} \| f \|_{L^2(M)}^2 \sup_{\gamma \in \Pi_0} \| g \|_{L^2(\mathcal{T}_{\lambda}(\gamma))}^2
		\\
			+~\varepsilon \lambda^{1/4} \| f \|_{L^2(M)}^2 \| g \|_{L^4(M)}^2
			+~C \| f \|_{L^2(M)}^2 \| g \|_{L^2(M)}^2			
		\end{multline}
	It is equivalent to show \eqref{hdeltageo} and \eqref{hdelta} with $\chi _\lambda(\sqrt{-\Delta_g}) e^{it_0 \sqrt{-\Delta_g}} f$
		in place of $\chi _\lambda(\sqrt{-\Delta_g}) f$ for some fixed $t_0$, because
	\begin{equation*}
		\|e^{-it_0 \sqrt{-\Delta_g}} f \|_{L^2(M)}=\| f \|_{L^2(M)}
	\end{equation*}
	Adapting \eqref{wave} gives
	\[
	\label{wave'}
		\chi_\lambda(\sqrt{-\Delta_g}) e^{it_0 \sqrt{-\Delta_g}} f = (2\pi)^{-1} \int \hat{\chi}(t) e^{-it\lambda} e^{i(t+t_0)\sqrt{-\Delta_g}}f \,dt
	\]
	
	Before proceeding, we prove Lemma \ref{neglap}.
	
	\begin{proof}[Proof of Lemma \ref{neglap}]
	If $\delta$ is small, we can apply a parametrix as follows.
	See, for example, Theorem ~4.1.2 in Sogge ~\cite{SogBook}.
	In appropriately chosen coordinate charts, the operator $\chi_\lambda(-\sqrt{-\Delta_0})$ is equal, modulo smoothing operators, to an operator with kernel
	\[
			\iint \hat{\chi}(t) e^{i[\varphi_0(x,y,\xi)-t p_0(y,\xi)-t\lambda]}q(t,x,y,\xi)\,dt d\xi
		\]
	Here $\varphi_0$ is smooth, $p_0$ is the principal symbol of $\sqrt{-\Delta_0}$, and $q$ is a symbol of type $(1, 0)$ and order zero.
	Since $p_0(y,\xi) \sim |\xi|$ and $\lambda \ge 1$,
	\[
			\Big| \frac{\partial}{\partial t} \Big( \varphi_0(x,y,\xi)-tp_0(y,\xi)-t\lambda \Big)\Big| = |p_0(y,\xi)+\lambda| \gtrsim 1+ |\xi|
		\]
	An integration by parts argument shows that for any positive integer $N$,
	\[
			\Big| \int \hat{\chi}(t) e^{i[\varphi_0(x,y,\xi)-t p_0(y,\xi)-t\lambda]}q(t,x,y,\xi)\,dt \Big| \le C_N (1+|\xi|)^{-N}
		\]
	So the kernel of $\chi_\lambda(-\sqrt{-\Delta_0})$ is uniformly bounded, independent of $\lambda$.
	\end{proof}
	
	We reduce the problem by following Smith-Sogge \cite{SSog}.
	For an operator $A$ from $M_0$ to $\mathbb{R} \times M_0$, define associated operators $I_\lambda(A)$ by
	\begin{equation*}
		I_\lambda(A)f(x) = \int \hat{\chi}(t) e^{-it\lambda} Af(t,x) \,dt
	\end{equation*}
	Here we can identify operators from $M$ to $\mathbb{R} \times M$ with operators from $M_0$ to $\mathbb{R}\times M_0$ whose kernels are supported in $M \times (\mathbb{R} \times M)$.
	Let $E_g$ be the operator given by
	\[
			E_g f(t,x)=\Big(e^{i(t+t_0)\sqrt{-\Delta_g}}f \Big)(x)
		\]
	Then we have
	\[
			I_\lambda(E_g)= 2\pi \, \chi_\lambda(\sqrt{-\Delta_g}) \circ e^{it_0 \sqrt{-\Delta_g}}
		\]
	We can rewrite \eqref{hdeltageo} and \eqref{hdelta} as
	\[
			\| I_\lambda(E_g) f \|_{L^4(\gamma \cap H_\delta)} \lesssim \lambda^{1/4} \| f \|_{L^2(M_0)}
		\]
	and
	\begin{multline*}
			\int_{H_\delta} | I_\lambda(E_g) f(x)|^2 |g(x)|^2 \,dx \le
			C_\varepsilon \lambda^{1/2} \| f \|_{L^2(M_0)}^2 \sup_{\gamma \in \Pi_0} \| g \|_{L^2(\mathcal{T}_{\lambda}(\gamma))}^2
		\\
			+~\varepsilon \lambda^{1/4} \| f \|_{L^2(M_0)}^2 \| g \|_{L^4(M)}^2
			+~C \| f \|_{L^2(M_0)}^2 \| g \|_{L^2(M)}^2
		\end{multline*}
	
	It suffices to write $E_g$ as a finite sum of operators, where for each operator $A$ in the sum, $I_\lambda(A)$ satisfies
	\begin{equation}
	\label{geoineq'}
			\| I_\lambda(A) f \|_{L^4(\gamma \cap H_\delta)} \lesssim \lambda^{1/4} \| f \|_{L^2(M_0)}
		\end{equation}
	and
	\begin{multline}
		\label{propineq'}
			\int_{H_\delta} | I_\lambda(A) f(x)|^2 |g(x)|^2 \,dx \le
			C_\varepsilon \lambda^{1/2} \| f \|_{L^2(M_0)}^2 \sup_{\gamma \in \Pi_0} \| g \|_{L^2(\mathcal{T}_{\lambda}(\gamma))}^2
		\\
			+~\varepsilon \lambda^{1/4} \| f \|_{L^2(M_0)}^2 \| g \|_{L^4(M)}^2
			+~C \| f \|_{L^2(M_0)}^2 \| g \|_{L^2(M)}^2
		\end{multline}
	If an operator $A$ has a kernel $K(t,x,y)$ which is uniformly bounded over the region
	\[
			\Big\{(t,x,y): t \in \text{supp } \hat{\chi}, x \in H_\delta, y \in M_0 \Big\}
		\]
	then the kernel of $I_\lambda(A)$ is uniformly bounded, independent of $\lambda$.
	In this case the estimates \eqref{geoineq'} and \eqref{propineq'} are trivial.
	In particular, this applies when $A$ is smoothing.
	
	Since $\partial M$ is strictly geodesically concave, there is a $c_0>0$ such that if $t_0 > 0$ is small then any unit speed broken geodesic $\gamma$ with $d(\gamma(0), \partial M) \le c_0 t_0^2$ must satisfy
	\begin{equation*}
		d(\gamma (t), \partial M) \ge c_0 t_0^2
	\end{equation*}
	for $\frac{1}{2}t_0 \le t \le 4 t_0$.
	Now define $\Omega$ to be the set of points $y$ in $M$ such that
	there is a unit speed broken geodesic $\gamma$ with $\gamma(0)=y$ and $d(\gamma(t_0+t), \partial M) \le 2\delta$ for some $t \in [-\delta, \delta]$.
	We assume that $2\delta < c_0 t_0^2$ and $\delta < \frac{1}{2}t_0$, which implies $d(\omega, \partial M) \ge c_0 t_0^2$.
	
	If the kernel of $E_g$ has a singularity at $(t, x, y)$ then there is a broken geodesic of length $t+t_0$ with endpoints at $x$ and $y$.
	So there is a smooth function $\alpha$ with support in $\Omega$ such that the kernel of the operator
	\begin{equation*}
		f \to E_g (1-\alpha) f
	\end{equation*}
	is smooth over the region $\{(t,x,y): t \in$ supp $\hat{\chi}, x \in H_\delta, y \in M_0 \}$.
	This reduces the problem to only considering $f$ with support in $\Omega$.
	
	Define an operator $E_0$ from $M_0$ to $\mathbb{R} \times M_0$ by
	\[
			E_0 f(t,x)=\Big(e^{i(t+t_0)\sqrt{-\Delta_0}}f \Big)(x)
		\]
	Let $\mathcal{R}$ be an operator from $M_0$ to $\mathbb{R} \times \partial M$ given by
	\begin{equation*}
		\mathcal{R}f=(E_0 f )\big{\vert}_{\mathbb{R} \times \partial M}
	\end{equation*}
	Let $\square _g = \partial_t^2 - \Delta _g$ and $\square _0 =\partial_t^2 - \Delta _0$.
	Let $W$ be the forward solution operator of the Dirichlet problem for $\square_g$, 
		mapping data on $\mathbb{R} \times \partial M$ which vanish for $t \le -t_0$ to functions on $\mathbb{R} \times M$.
	That is, the equation $u=Wh$ means $u$ solves
	\[
		\left\{ 
		\begin{array}{l l}
			\square _g u & = 0\\
			u & =0 \quad \text{for } t \le -t_0\\
			u \vert _{\mathbb{R} \times \partial M} & =h\\
		\end{array}
		\right.
	\]
	Recall we are assuming $\delta < \frac{1}{2}t_0$.
	Now over $[\frac{1}{2}\delta, \delta] \times M$, for $f$ supported in $\Omega$,
	\begin{equation*}
		E_g f=E_0 f-W\mathcal{R}_+ f
	\end{equation*}
	where $\mathcal{R}_+$ is $\mathcal{R}$ smoothly cutoff to $t$ in $[-t_0, t_0]$.
	
	We can break up the cotangent bundle of $\mathbb{R} \times \partial M$ into three time-independent conic regions.
	These are the elliptic and hyperbolic regions where the Dirichlet problem is elliptic and hyperbolic, respectively, and the glancing region which is the region between them.
	We can break up the identity operator into a sum of time-independent conic pseudodifferential cutoffs as
	\begin{equation*}
		I=\Pi_e+\Pi_h+\Pi_g
	\end{equation*}
	where $\Pi_e$ and $\Pi_h$ are essentially supported strictly inside the elliptic and hyperbolic regions, respectively,
		and $\Pi_g$ is essentially supported in a small conic set about the glancing region.
	Then over $[\frac{1}{2}\delta, \delta] \times M$,
	\begin{equation*}
		E_g f=E_0 f-W\Pi_e\mathcal{R}_+ f-W\Pi_h\mathcal{R}_+ f-W\Pi_g\mathcal{R}_+ f
	\end{equation*}
	
	The operator $I_\lambda(E_0)$ is equal to $2 \pi \, \chi_\lambda(\sqrt{-\Delta_0}) \circ e^{it_0 \sqrt{-\Delta_0}}$,
		so it satisfies \eqref{geoineq'} and \eqref{propineq'} by Theorems ~\ref{bgtth} and \ref{soggeth}.

	The projection of any characteristic direction of $\square_g$ onto $T^*(\mathbb{R}\times \partial M)$ is contained in the hyperbolic or glancing regions, so $W\Pi_e\mathcal{R}_+$ is smoothing.
	This implies that $I_\lambda(W\Pi_e\mathcal{R}_+)$ satisfies \eqref{geoineq'} and \eqref{propineq'}.
	
	On the essential support of $\Pi_h$, we can solve the forward Dirichlet problem for $\square_g$ locally, modulo smoothing operators,
		on an open set in $\mathbb{R} \times M_0$ around $\mathbb{R} \times \partial M$.
	This gives a positive constant $t_1$ and an operator $\tilde{W}$ from $\mathbb{R} \times \partial M$ to $\mathbb{R} \times M_0$ such that for any $v$ supported by $t$ in $[-t_1, t_1]$,
		we have that $\square_0 \tilde{W}v$ is smooth over $[-2t_1, 2t_1] \times M_0$ and $(W-\tilde{W}) \Pi_h v$ is smooth over $\mathbb{R} \times M$.
	
	We can assume $t_0 \le t_1$ and define operators $J_1$ and $J_2$ by
	\[
			J_1 f = \Big( \tilde{W} \Pi_h \mathcal{R}_+ f \Big) \Big\vert _{t=-t_0}
		\]
	\[
			J_2 f = (-\Delta_0)^{-1/2} \bigg(\Big(\partial_t \tilde{W} \Pi_h \mathcal{R}_+ f \Big) \Big\vert _{t=-t_0} \bigg)
		\]
	These are Fourier integral operators of order zero associated to the relation of reflection about $\del M$.
	
	Define operators $C_0$ and $S_0$ from $M_0$ to $\mathbb{R} \times M_0$ by
	\[
			C_0 f(t,x)=\Big(\cos\big((t+t_0)\sqrt{-\Delta_0}\big)f\Big)(x)
		\]
	and
	\[
			S_0 f(t,x)=\Big(\sin\big((t+t_0)\sqrt{-\Delta_0}\big)f\Big)(x)
		\]
	We can write $W\Pi_h\mathcal{R}_+ f$, modulo smoothing operators, as $C_0 J_1 f + S_0 J_2 f$.
	By the $L^2$ continuity of $J_1$ and $J_2$, it remains to show that $I_\lambda(C_0)$ and $I_\lambda(S_0)$ satisfy \eqref{geoineq'} and \eqref{propineq'}.
	Define an operator $\tilde{E}_0$ from $M_0$ to $\mathbb{R} \times M_0$ by
	\[
			\tilde{E}_0 f(t,x)=\Big(e^{-i(t+t_0)\sqrt{-\Delta_0}}f \Big)(x)
		\]
	Since $I_\lambda(E_0)$ satisfies \eqref{geoineq'} and \eqref{propineq'}, it suffices to show that the same is true for $I_\lambda(\tilde{E}_0)$.
	This follows from Lemma \ref{neglap}, completing the argument for the term $W\Pi_h\mathcal{R}_+ f$.
	
	Now we break up $\Pi_g$ into a finite sum of pseudodifferential cutoffs, each essentially supported in a suitably small conic neighborhood of a glancing ray.
	This breaks up $W\Pi_g\mathcal{R}_+f$ into a finite sum and the Melrose-Taylor parametrix ~\cite{MTBook} can be applied to each term.
	We will use coordinates for $M_0$, chosen so that $M$ is given by $x_2 > 0$.
	Then each term in this sum can be written, modulo smoothing operators, in the form $GKf$, where $K$ is a Fourier integral operator of order zero, compactly supported on both sides,
		and $G$ is an operator from $\mathbb{R}^2$ to $\mathbb{R}^3$ with kernel
	\[
			\int e^{i\theta(x,\xi)+it\xi_1-iy\cdot \xi} \Big( A_+ \big( \zeta(x,\xi) \big) a(x,\xi) + A_+' \big( \zeta(x,\xi) \big) b(x,\xi) \Big) \frac{Ai}{A_+}\big( \zeta_0(\xi) \big) \,d\xi
		\]
	The functions $a$ and $b$ are symbols of type $(1,0)$ and order $1/6$ and $-1/6$, respectively,
		and both are supported by $x$ in a small ball about the origin and by $\xi$ is in a small conic neighborhood of the $\xi_1$-axis.
	Also $Ai$ is the Airy function, and $A_+$ is given by $A_+(z)=Ai(e^{-\frac{2}{3}\pi i} z)$.
	The function $\zeta_0$ is defined by $\zeta_0(\xi)=-\xi_1^{-1/3}\xi_2$, and the phases $\theta$ and $\zeta$ are real, smooth, and homogeneous in $\xi$ of degree 1 and $2/3$, respectively, with
	\begin{equation}
		\label{bczeta}
			\zeta\big((x_1,0),\xi\big)=\zeta_0(\xi) \quad \text{and} \quad \frac{\partial\zeta}{\partial x_2}\big((x_1,0),\xi\big) < 0
		\end{equation}
	Let $\langle \; , \; \rangle_x$ be the inner product given by $g$.
	In the region $\zeta(x,\xi) \le 0$, the functions $\theta$ and $\zeta$ satisfy
	\begin{equation}
	\label{eikonal0}
		\left\{ 
		\begin{array}{l l}
			\xi_1^2 - \langle d_x\theta, d_x\theta \rangle_x + \zeta \langle d_x\zeta, d_x\zeta \rangle_x =0 \\
			\langle d_x\theta, d_x\zeta \rangle_x =0
		\end{array}
		\right.
	\end{equation}
	Also, $\theta$ and $\zeta$ satisfy these equations to infinite order at $x_2=0$ in the region $\zeta(x,\xi) > 0$.
	
	Fix a small $r>0$ and define the set
	\[
			S_r= \Big\{ x \in \mathbb{R}^2: |x| \le r, x_2 \ge 0 \Big\}
		\]
	We identify $S_r$ with a subset of $M$.
	For an operator $A$ from $\mathbb{R}^2$ to $\mathbb{R}^3$, define associated operators $I_\lambda(A)$ by
	\begin{equation*}
		I_\lambda(A)f(x) = \int  \hat{\chi}(t) e^{-it\lambda} Af(t,x) \,dt
	\end{equation*}
	By the $L^2$ continuity of $K$ it suffices to show that $I_\lambda(G)$ has the following properties.
	For a broken geodesic $\gamma$ in $S_r$ of unit length and for $f$ with fixed compact support, we need to show that
	\begin{equation*}
			\| I_\lambda(G) f \|_{L^4(\gamma)} \lesssim \lambda^{1/4} \| f \|_{L^2(\mathbb{R}^2)}
		\end{equation*}
	We also need to show that for any $\varepsilon>0$, there is a constant $C_\varepsilon$ such that for $f$ with fixed compact support,
	\begin{multline*}
			\int_{S_r} | I_\lambda(G) f(x)|^2 |g(x)|^2 \,dx \le
			C_\varepsilon \lambda^{1/2} \| f \|_{L^2(\mathbb{R}^2)}^2 \sup_{\gamma \in \Pi_0} \| g \|_{L^2(\mathcal{T}_{\lambda}(\gamma))}^2
		\\
			+~\varepsilon \lambda^{1/4} \| f \|_{L^2(\mathbb{R}^2)}^2 \| g \|_{L^4(\mathbb{R}^2)}^2
			+~C \| f \|_{L^2(\mathbb{R}^2)}^2 \| g \|_{L^2(\mathbb{R}^2)}^2
		\end{multline*}
	It suffices to write $G$ as a finite sum of operators, where for each operator $A$ in the sum and for $f$ with fixed compact support, $I_\lambda(A)$ satisfies
	\begin{equation}
	\label{coord1}
			\| I_\lambda(A) f \|_{L^4(\gamma)} \lesssim \lambda^{1/4} \| f \|_{L^2(\mathbb{R}^2)}
		\end{equation}
	and
	\begin{multline}
	\label{coord2}
			\int_{S_r} | I_\lambda(A) f(x)|^2 |g(x)|^2 \,dx \le
			C_\varepsilon \lambda^{1/2} \| f \|_{L^2(\mathbb{R}^2)}^2 \sup_{\gamma \in \Pi_0} \| g \|_{L^2(\mathcal{T}_{\lambda}(\gamma))}^2
		\\
			+~\varepsilon \lambda^{1/4} \| f \|_{L^2(\mathbb{R}^2)}^2 \| g \|_{L^4(\mathbb{R}^2)}^2
			+~C \| f \|_{L^2(\mathbb{R}^2)}^2 \| g \|_{L^2(\mathbb{R}^2)}^2
		\end{multline}
	If an operator $A$ has a kernel $K(t,x,y)$ which is uniformly bounded over compact subsets of
	\[
			\Big\{(t,x,y): t \in \text{supp } \hat{\chi}, x \in S_r, y \in \mathbb{R}^2 \Big\}
		\]
	then the kernel of $I_\lambda(A)$ is uniformly bounded, independent of $\lambda$, over compact subsets of $S_r \times \mathbb{R}^2$.
	In this case the estimates \eqref{coord1} and \eqref{coord2} are trivial.
	In particular, this applies when $A$ is smoothing.
	
	Let $\rho$ be a smooth function with $\rho(s)=0$ for $s \ge -1$ and $\rho(s)=1$ for $s \le -2$.
	Following Zworski ~\cite{Z}, we break up $G$ into $G_m+G_d$, where the kernel of $G_m$ is
	\[
			\int e^{i\theta(x,\xi)+it\xi_1-iy\cdot\xi} \Big( (\rho A_+) \big( \zeta(x,\xi) \big) a(x,\xi) + (\rho A_+)' \big( \zeta(x,\xi) \big) b(x,\xi) \Big) \frac{Ai}{A_+}\big( \zeta_0(\xi) \big) \,d\xi
		\]
	and the kernel of $G_d$ is
	\[
			\int e^{i\theta(x,\xi)+it\xi_1-iy\cdot \xi} q(x,\xi) \,d\xi
		\]
	Here we have
	\begin{equation}
		\label{diffsym}
			q(x,\xi) = \Big( \big( (1-\rho) A_+ \big) \big( \zeta(x,\xi) \big) a(x,\xi) + \big( (1-\rho) A_+ \big)' \big( \zeta(x,\xi) \big) b(x,\xi) \Big) \frac{Ai}{A_+}\big( \zeta_0(\xi) \big)
		\end{equation}
	We will refer to $G_m$ as the main term and to $G_d$ as the diffractive term.
	
	Define an operator $\tilde{G}_m$ with kernel
	\[
			\int e^{i\theta(x,\xi)+it\xi_1-iy\cdot\xi} \Big( (\rho A_+) \big( \zeta(x,\xi) \big) a(x,\xi) + (\rho A_+)' \big( \zeta(x,\xi) \big) b(x,\xi) \Big) \,d\xi
		\]
	Then to control $I_\lambda(G_m)$, it suffices to show that $I_\lambda(\tilde{G}_m)$ satisfies \eqref{coord1} and \eqref{coord2}, because
	\[
			| \frac{Ai}{A_+}(s) | \le 2 \quad \text{for } s \in \mathbb{R}
		\]
	By stationary phase,
	\[
			\widehat{(\rho A_+)}(s)= 2\pi \, e^{i\frac{1}{3}s^3} \Psi_+(s)
		\]
	where $\Psi_+$ is smooth and satisfies
	\[
			\Big| \frac{d^k}{ds^k} \Psi_+(s) \Big| \le C_k
		\]
	Applying the Fourier inversion formula and changing variables gives
	\[
			(\rho A_+)(\zeta)=\int e^{i(s \xi_1^{-2/3} \zeta+\frac{1}{3}s^3 \xi_1^{-2})} \xi_1^{-2/3} \Psi_+(\xi_1^{-2/3}s) ds
		\]
	Similarly,
	\[
			(\rho A_+)'(\zeta)=\int e^{i(s \xi_1^{-2/3} \zeta+\frac{1}{3}s^3 \xi_1^{-2})} s \xi_1^{-4/3} \Psi_+(\xi_1^{-2/3}s) ds
		\]
	So the kernel of $\tilde{G}_m$ is
	\begin{multline*}
			\iint e^{i[\theta(x,\xi)+t\xi_1+s \xi_1^{-2/3} \zeta(x,\xi)+\frac{1}{3}s^3 \xi_1^{-2}-y\cdot \xi]}
			\\
			\times \xi_1^{-2/3} \Psi_+(\xi_1^{-2/3}s) \Big( a(x,\xi) + s \xi_1^{-2/3} b(x,\xi) \Big)  \,ds d\xi
		\end{multline*}
	Here the symbol
	\[
			\xi_1^{-2/3} \Psi_+(\xi_1^{-2/3}s) \Big( a(x,\xi) + s \xi_1^{-2/3} b(x,\xi) \Big)
		\]
	is of type $(2/3,1/3)$ and order $-1/2$ on $\mathbb{R}_x^2 \times \mathbb{R}_{s,\xi}^3$.
	Let $\psi_0$ be the function
	\[
			\psi_0(x,t,\xi,s)=\theta(x,\xi)+t\xi_1+s \xi_1^{-2/3} \zeta(x,\xi)+\frac{1}{3}s^3 \xi_1^{-2}
		\]
	We need to prove the following.
	\begin{Lemma}
		\label{Labnon}
			Fix $B \in S_{2/3,1/3}^{-1/2}(\R_x^2 \times \R_{s,\xi}^3)$ supported by $x$ in a small neighborhood of the origin and $\xi$ in a small conic neighborhood of the $\xi_1$-axis.
			Define an operator $V_B$ with kernel
			\[
				\iint e^{i \psi_0(x,t,\xi,s)-iy \cdot \xi} B(x,\xi,s) \,dsd\xi
			\]
			Then for any broken geodesic $\gamma$ in $S_r$ of unit length and for $f$ with fixed compact support, the operators $I_\lambda(V_B)$ satisfy
		\begin{equation}
			\label{abnon1}
				\| I_\lambda(V_B) f \|_{L^4(\gamma)} \lesssim \lambda^{1/4} \| f \|_{L^2(\mathbb{R}^2)}
			\end{equation}
		Also for any $\varepsilon>0$ and for $f$ with fixed compact support, there is a constant $C_\varepsilon$ such that the operators $I_\lambda(V_B)$ satisfy
		\begin{multline}
			\label{abnon2}
				\int_{S_r} | I_\lambda(V_B) f(x)|^2 |g(x)|^2 \,dx \le
				C_\varepsilon \lambda^{1/2} \| f \|_{L^2(\mathbb{R}^2)}^2 \sup_{\gamma \in \Pi_0} \| g \|_{L^2(\mathcal{T}_{\lambda}(\gamma))}^2
			\\
				+~\varepsilon \lambda^{1/4} \| f \|_{L^2(\mathbb{R}^2)}^2 \| g \|_{L^4(\mathbb{R}^2)}^2
				+~C \| f \|_{L^2(\mathbb{R}^2)}^2 \| g \|_{L^2(\mathbb{R}^2)}^2
			\end{multline}
		\end{Lemma}
	
	We have seen that the estimates for the main term will follow from Lemma ~\ref{Labnon}.
	Before proving Lemma ~\ref{Labnon}, we will show that it also implies the estimates for the diffractive term.
	First, we will show that for $x$ in $S_r$ and for $\xi$ in a small conic neighborhood of the $\xi_1$-axis, the symbol $q(x,\xi)$ defined by \eqref{diffsym} can be written as
	\begin{equation}
		\label{qhform}
			q(x,\xi) = h\big(x,\xi,\zeta(x,\xi)\big)
		\end{equation}
	where
	\[
			\Big| \partial_\xi^\alpha \partial_\zeta^j\partial_{x_1}^k\partial_{x_2}^\ell h(x,\xi_1,\zeta) \Big| \le C_{\alpha,j,k,\ell} \, \xi_1^{1/6-|\alpha|+2\ell/3} e^{-cx_2^{3/2}\xi_1-\frac{1}{2}|\zeta|^{3/2}}
		\]
	for some $c>0$.
	Fix $\eps >0$.
	Then
	\[
			\Big| \del_\zeta^k \big( (1-\rho) A_+ \big)(\zeta) \Big| \le C_{\eps,k} \, e^{(\frac{2}{3}+\eps)|\zeta|^{3/2}}
		\]
	If $\eps$ is small, then it suffices to show that, in the region $\zeta(x,\xi) \ge -2$,
	\[
			\frac{Ai}{A_+} \big(\zeta_0(\xi)\big)=H\big(x,\xi_1,\zeta(x,\xi)\big)
		\]
	where
	\begin{equation}
		\label{hform}
			\Big| \partial_{\xi_1}^m \partial_\zeta^j\partial_{x_1}^k\partial_{x_2}^\ell H(x,\xi_1,\zeta) \Big| \le C_{m,j,k,\ell} \, \xi_1^{-m+2\ell/3} e^{-cx_2^{3/2}\xi_1-(\frac{4}{3}-\eps)|\zeta|^{3/2}}
		\end{equation}
	By \eqref{bczeta}, there is a $c>0$ such that
	\[
			\zeta_0(\xi) \ge \zeta(x,\xi) + c x_2 \xi_1^{2/3}
		\]  
	In the region $\zeta(x,\xi) \ge -2$, the asymptotics of the Airy functions now yield
	\begin{equation}
		\label{airyas}
			\Big| \Big(\frac{Ai}{A_+}\Big)^{(m)}\big(\zeta_0(\xi)\big) \Big| \le C_{\eps,m} e^{-cx_2^{3/2}\xi_1-(\frac{4}{3}-\eps)|\zeta(x,\xi)|^{3/2}}
		\end{equation}
	Define a new variable
	\[
			\tau(x,\xi)=\xi_1^{1/3} \zeta(x,\xi)
		\]
	When $x_2=0$, we have $\tau=-\xi_2$.
	It follows that we can write $\xi_2=\sigma(x,\xi_1,\tau)$, where $\sigma$ is homogeneous of degree 1 in $(\xi_1,\tau)$.
	Now we define
	\[
			H(x,\xi_1,\zeta)=\frac{Ai}{A_+}\big(-\xi_1^{-1/3}\sigma(x,\xi_1,\xi_1^{1/3}\zeta)\big)
		\]
	To prove \eqref{hform} it suffices to show that
	\begin{multline}
		\label{hred}
			\Big| \partial_{\xi_1}^m \partial_\tau^j\partial_{x_1}^k\partial_{x_2}^\ell \frac{Ai}{A_+}\big(-\xi_1^{-1/3}\sigma(x,\xi_1,\tau)\big) \Big| 
			\\
			\le C_{m,j,k,\ell} \, \xi_1^{-m-j+2\ell/3} e^{-cx_2^{3/2}\xi_1-(\frac{4}{3}-\eps)|\tau|^{3/2}\xi_1^{-1/2}}
		\end{multline}
	If $x_2=\tau=0$, then $\sigma(x,\xi_1,\tau)=0$.
	So the homogeneity of $\sigma$ implies that
	\[
			\Big| \partial_{\xi_1}^m \partial_\tau^j\partial_{x_1}^k \big(-\xi_1^{-1/3}\sigma(x,\xi_1,\tau)\big) \Big| \le C_{m,j,k} (x_2 \xi_1^{2/3}+\xi_1^{-1/3}|\tau|)\xi_1^{-m-j}
		\]
	Together with \eqref{airyas}, this implies \eqref{hred} when $\ell=0$.
	It also follows for other values of $\ell$ because differentiating with respect to $x_2$ in \eqref{hred} is similar to multiplying by a symbol of type $(1,0)$ and order $2/3$.
	Then \eqref{hform} follows.
	
	Now we can write the Fourier transform of $h(x,\xi,\zeta)$ in the $\zeta$-variable as
	\[
			\int e^{-is\zeta} q_0(x,\xi,\zeta) \,d\zeta = 2\pi \, e^{i\frac{1}{3}s^3} w(x,\xi,s)
		\]
	where, for any $N>0$,
	\[
			\Big| \partial_\xi^\alpha \partial_s^j \partial_{x_1}^k \partial_{x_2}^\ell w(x,\xi,s) \Big| \le C_{\alpha,j,k,\ell} \, \xi_1^{1/6-|\alpha|+2\ell/3} e^{-cx_2^{3/2}\xi_1} (1+s)^{-N}
		\]
	Applying the Fourier inversion formula and changing variables gives
	\[
			q_0(x,\xi,\zeta)=\int e^{i(s\xi_1^{-2/3}\zeta+\frac{1}{3}s^3\xi_1^{-2})} \xi_1^{-2/3}w(x,\xi,\xi_1^{-2/3}s) \,ds
		\]
	Now we can write the kernel of $G_d$ as
	\[
			\iint e^{i\psi_0(x,t,\xi,s)-iy \cdot \xi} c(x,\xi,s) \,dsd\xi
		\]
	where $c$ is supported by $x$ in a small ball and by $\xi$ in a small conic neighborhood of the $\xi_1$ axis and satisfies
	\[
			\Big| \partial_\xi^\alpha \partial_s^j \partial_{x_1}^k \partial_{x_2}^\ell c(x,\xi,s) \Big| \le C_{\alpha,j,k,\ell} \, \xi_1^{-1/2-|\alpha|-2j/3+2\ell/3} e^{-cx_2^{3/2}\xi_1} (1+\xi_1^{-2/3}s)^{-N}
		\]
	for any $N>0$.
	In particular,
	\begin{equation*}
		x_2^j \partial_{x_2}^k c(x,\xi,s) \in S_{2/3, 1/3}^{-1/2+2(k-j)/3}(\mathbb{R}_{x_1} \times \mathbb{R}_{\xi,s}^3)
	\end{equation*}
	uniformly over $x_2$.
	
	Let $v$ be in $C_0^\infty(\R^2)$ have small support and satisfy $c(x,\xi,s)=v(x)c(x,\xi,s)$.
	Then we have
	\begin{equation*}
		c(x,\xi,s)=v(x)c(x_1,0,\xi,s)+\int_0^{x_2}v(x)\partial_{x_2}c(x_1,\sigma,\xi,s) \,d\sigma
	\end{equation*}
	So we can write $G_d=A_d+B_d$ where the kernel of $A_d$ is
	\[
			\iint e^{i\psi_0(x,t,\xi,s)-iy \cdot \xi} v(x)c(x_1,0,\xi,s) \,dsd\xi
		\]
	The symbol $v(x)c(x_1,0,\xi,s)$ is of type $(2/3, 1/3)$ and order $-1/2$.
	So $I_\lambda(A_d)$ satisfies \eqref{coord1} and \eqref{coord2} by Lemma ~\ref{Labnon}.
	
	The kernel of $I_\lambda(B_d)$ is
	\[
			\int_0^{x_2} \iiint \hat{\chi}(t) e^{-it\lambda +i\psi_0(x,t,\xi,s)-iy \cdot \xi} v(x) \partial_{x_2} c(x_1,\sigma,\xi,s) \,dsd\xi dtd\sigma
		\]
	Let $\beta$ be a smooth function supported in $[1/3, 3]$ with $\beta=1$ on $[1/2, 2]$.
	Define operators $B_\lambda$ with kernels
	\[
			\int_0^{x_2} \iint e^{i\psi_0(x,t,\xi,s)-iy \cdot \xi} \beta\Big(\frac{\xi_1}{\lambda}\Big) v(x) \partial_{x_2}c(x_1,\sigma,\xi,s) \,dsd\xi d\sigma
		\]
	The kernel of $I_\lambda(B_\lambda)$ is
	\[
			\int_0^{x_2} \iiint \hat{\chi}(t) e^{-it\lambda+i\psi_0(x,t,\xi,s)-iy \cdot \xi} \beta \Big( \frac{\xi_1}{\lambda} \Big) v(x) \partial_{x_2}c(x_1,\sigma,\xi,s) \,dsd\xi dtd\sigma
		\]
	Since $\partial_t \psi_0 = \xi_1$, an integration by parts argument shows that $I_\lambda(B_d)$ differs from $I_\lambda(B_\lambda)$
		by an operator whose kernel is uniformly bounded, independent of $\lambda$.
	So it suffices to prove $I_\lambda(B_\lambda)$ satisfies \eqref{coord1} and \eqref{coord2}.
	Let
	\[
			P_{\sigma,\lambda}(x,\xi,s)=v(x) \beta \Big( \frac{\xi_1}{\lambda} \Big) \partial_{x_2} c(x_1,\sigma,\xi,s)
		\]
	Then
	\[
			| I_\lambda(B_\lambda) f | \le \int \Big| \iiiint \hat{\chi}(t) e^{-it\lambda +i\psi_0(x,t,\xi,s)-iy \cdot \xi} P_{\sigma,\lambda}(x,\xi,s) f(y) \,dydsd\xi dt \Big| d\sigma
		\]
	Define operators $B_{\sigma,\lambda}$ by
	\[
			B_{\sigma,\lambda}f(t,x)= \iiint e^{i\psi_0(x,t,\xi,s)-iy \cdot \xi} \lambda^{-2/3}(1+\lambda^{4/3}\sigma^2) P_{\sigma,\lambda}(x,\xi,s) f(y) \,dydsd\xi
		\]	
	By Minkowski's integral inequality and H\"older's inequality,
	\begin{equation}
		\label{supr1}
			\| I_\lambda(B_\lambda) f \|_\Lp{2}{\gamma} \leC \sup_\sigma \| I_\lambda(B_{\sigma,\lambda}) f \|_\Lp{2}{\gamma}
		\end{equation}
	Also
	\begin{equation}
		\label{supr2}
			\int_{S_r} | I_\lambda(B_\lambda) f(x)|^2 |g(x)|^2 \,dx \leC \sup_\sigma \int_{S_r} | I_\lambda(B_{\sigma,\lambda}) f(x)|^2 |g(x)|^2 \,dx
		\end{equation}
	The amplitudes
	\[
			\lambda^{-2/3}(1+\lambda^{4/3}\sigma^2) P_{\sigma,\lambda}(x,\xi,s)
		\]
	are symbols of type $(2/3, 1/3)$ and order $-1/2$ over $\mathbb{R}_x^2 \times \mathbb{R}_{\xi,s}^3$, uniformly in $\sigma$ and $\lambda$.
	By Lemma ~\ref{Labnon}, the operators $I_\lambda(B_{\sigma,\lambda})$ satisfy \eqref{coord1} and \eqref{coord2}, uniformly in $\sigma$.
	Then $I_\lambda(B_\lambda)$ satisfies \eqref{coord1} and \eqref{coord2} because of \eqref{supr1} and \eqref{supr2}.
	So Lemma \ref{Labnon} will imply the estimates for the diffractive term.
	
	To prove Lemma \ref{Labnon}, note that $V_B$ is a Fourier integral operator of type $(2/3,1/3)$ and order zero associated to the canonical relation $\mathcal{C}$ given by
	\[
			\mathcal{C}=\Big\{ \Big(x,t, \nabla_x \psi_0(x,t,\xi,s), \xi_1; \nabla_\xi \psi_0(x,t,\xi,s), \xi \Big) : \zeta(x,\xi)=-s^2 \xi_1^{-4/3} \}
		\]
	Let $\mathcal{C}_0$ be the restriction of $\mathcal{C}$ to $t=0$.
	It was shown in the proof of Lemma A.2 of Smith-Sogge \cite{SSog2} that $\mathcal{C}_0$ is the graph of a canonical transformation.
	
	The projection of $\mathcal{C}$ onto $T^*(\mathbb{R}_{x,t}^3)$ is contained in the characteristic variety of $\square_0$, because of \eqref{eikonal0}.
	So the canonical relation $\mathcal{C}\circ\mathcal{C}_0^{-1}$ is the flowout, under the bicharacteristic flow of $\square_0$, of a conical subset of the diagonal at $t=0$.
	By the Lax construction, $\mathcal{C}\circ\mathcal{C}_0^{-1}$ can be parametrized by a phase function
	\[
			\varphi(t,x,\xi)-y\cdot\xi
		\]
	where $\varphi$ satisfies
	\begin{equation}
		\label{eikonal}
			\varphi(0,x,\xi)=x \cdot \xi \quad \text{and} \quad \frac{\partial \varphi}{\partial t} = p_0\Big(x, \frac{\partial \varphi}{\partial x}\Big)
		\end{equation}
	Here $p_0$ is the principal symbol of $\sqrt{-\Delta_0}$, that is
	\[
			p_0(x,\xi)=\sqrt{\sum g^{jk}(x)\xi_j\xi_k}
		\]
	Since $\varphi(t,x,\xi)-y\cdot \xi$ parametrizes $\mathcal{C}\circ\mathcal{C}_0^{-1}$, it follows that for small $t$,
	\begin{equation}
		\label{flow}
			y=\varphi_\xi'(t,x,\xi) \quad \text{implies} \quad t=d_0(x,y)
		\end{equation}
	
	Now let $J_0$ and $K_0$ be Fourier integral operators of order zero, compactly supported on both sides, associated to the canonical relations $\mathcal{C}_0^{-1}$ and $\mathcal{C}_0$, respectively,
		such that $V_B \circ J_0 \circ K_0$ differs from $V_B$ by a smoothing operator.
	To prove Lemma \ref{Labnon}, we need to show that $I_\lambda(V_B \circ J_0 \circ K_0)$ satisfies \eqref{coord1} and \eqref{coord2}.
	By the $L^2$ continuity of $K_0$, it suffices to show instead that $I_\lambda(V_B \circ J_0)$ satisfies \eqref{coord1} and \eqref{coord2}.
	Here $V_B \circ J_0$ is a Fourier integral operator of type $(2/3,1/3)$ and order zero, associated to the canonical relation $\mathcal{C}\circ\mathcal{C}_0^{-1}$.
	So its kernel, modulo smoothing operators, is of the form
	\[
			\int e^{i[\varphi(t,x,\xi)-y\cdot\xi]}a(t,x,\xi) \, d\xi
		\]
	where $a$ is a symbol of type $(2/3,1/3)$ and order zero on $\mathbb{R}_{t,x}^3 \times \mathbb{R}_\xi^2$.
	To show $I_\lambda(V_B \circ J_0)$ satisfies \eqref{coord2}, it now suffices to prove the following lemma.
	
	\begin{Lemma}
	\label{Lredsog}
		Fix $a \in S_{2/3,1/3}^0(\mathbb{R}_{t,x}^3 \times \mathbb{R}_\xi^2)$, supported by $x$ in a small neighborhood of $S_r$.
		Define an operator $U_a$ by
		\begin{equation*}
			U_a f = \iint e^{i\varphi(t,x,\xi)-iy\cdot\xi} a(t,x,\xi) f(y) \, d\xi dy
		\end{equation*}
		For any $\varepsilon > 0$ there is a constant $C_\varepsilon$ such that for $f$ with fixed compact support,
		\begin{multline*}
			\int_{S_r} | I_\lambda(U_a) f(x)|^2 |g(x)|^2 \,dx \le
			C_\varepsilon \lambda^{1/2} \| f \|_{L^2(\mathbb{R}^2)}^2 \sup_{\gamma \in \Pi_0} \| g \|_{L^2(\mathcal{T}_{\lambda}(\gamma))}^2
		\\
			+~\varepsilon \lambda^{1/4} \| f \|_{L^2(\mathbb{R}^2)}^2 \| g \|_{L^4(\mathbb{R}^2)}^2
			+~C \| f \|_{L^2(\mathbb{R}^2)}^2 \| g \|_{L^2(\mathbb{R}^2)}^2
		\end{multline*}
	\end{Lemma}
	
	We will prove Lemma \ref{Lredsog} in the next section.
	This will complete the proof of Theorem \ref{thlink}.
	The next lemma will show that $I_\lambda(V_B \circ J_0)$ satisfies \eqref{coord1}.
	
	\begin{Lemma}
	\label{Lredbgt}
		Fix $a \in S_{2/3,1/3}^0(\mathbb{R}_{t,x}^3 \times \mathbb{R}_\xi^2)$, supported by $x$ in a small neighborhood of $S_r$.
		Define an operator $U_a$ by
		\begin{equation*}
			U_a f = \iint e^{i\varphi(t,x,\xi)-iy\cdot\xi} a(t,x,\xi) f(y) \, d\xi dy
		\end{equation*}
		For any broken geodesic $\gamma$ in $S_r$ of unit length, and for $f$ with fixed compact support,
		\begin{equation*}
				\| I_\lambda(U_a) f \|_{L^4(\gamma)} \lesssim \lambda^{1/4} \| f \|_{L^2(\mathbb{R}^2)}
			\end{equation*}
	\end{Lemma}
	
	We will prove Lemma \ref{Lredbgt} in the fourth section.
	This will complete the proof of Theorem \ref{thbgt}.
	
\section{End of Proof of Theorem 1.3}
	
	To prove Theorem ~\ref{thlink}, it remains to prove Lemma \ref{Lredsog}.	
	This will be a consequence of the following variant.
	To state it, let $\eta(x,y)$ be in $C_0^\infty(\R^2 \times \R^2)$ be supported by $x$ and $y$ in a small neighborhood of $S_r$ satisfying $\frac{1}{2}\delta \le d_0(x,y) \le \delta$.
	Also assume $\eta(x,y)=1$ when $x$ is in a small neighborhood of $S_r$ and $d_0(x,y)$ is in an open neighborhood of the support of $\hat{\chi}$.

	\begin{Lemma}
	\label{Lswitched}
		Fix $b \in S_{2/3,1/3}^0(\mathbb{R}_{t,y}^3 \times \mathbb{R}_\xi^2)$.
		Define an operator $T_b$ by
		\begin{equation*}
			T_b f = \iint e^{i\varphi(t,x,\xi)-iy\cdot\xi} \eta(x,y) b(t,y,\xi) f(y)\, d\xi dy
		\end{equation*}
		For any $\varepsilon > 0$ there is a constant $C_\varepsilon$ such that for $f$ with fixed compact support,
		\begin{multline*}
		\label{redineq}
			\int_{S_r} | I_\lambda(T_b) f(x)|^2 |g(x)|^2 \,dx \le
			C_\varepsilon \lambda^{1/2} \| f \|_{L^2(\mathbb{R}^2)}^2 \sup_{\gamma \in \Pi_0} \| g \|_{L^2(\mathcal{T}_{\lambda}(\gamma))}^2
		\\
			+~\varepsilon \lambda^{1/4} \| f \|_{L^2(\mathbb{R}^2)}^2 \| g \|_{L^4(\mathbb{R}^2)}^2
			+~C \| f \|_{L^2(\mathbb{R}^2)}^2 \| g \|_{L^2(\mathbb{R}^2)}^2
		\end{multline*}
	\end{Lemma}
		
	Using Lemma \ref{Lswitched}, we can prove Lemma ~\ref{Lredsog}.
	
	\begin{proof}[Proof of Lemma \ref{Lredsog}]
	Fix a symbol $a \in S_{2/3,1/3}^0(\mathbb{R}_{t,x}^3 \times \mathbb{R}_\xi^2)$.
	We may assume that $\big(1-\eta(x,y)\big)a(t,x,\xi)$ vanishes on a neighborhood of the set 
	\[
			\Sigma_0=\Big\{ (t,x,y,\xi): t=d_0(x,y) \Big\}
		\]
	We can make this assumption because $I_\lambda(U_a)$ only depends on $t$ in the support of $\hat{\chi}$.
	The kernel of $U_a$ is
	\begin{equation*}
		 \int e^{i\varphi(t,x,\xi)-iy \cdot \xi} a(t,x,\xi) \,d\xi
	\end{equation*}
	Define an operator $D_a$ with kernel
	\[
			\int e^{i\varphi(t,x,\xi)-iy \cdot \xi} \eta(x,y) a(t,x,\xi) \,d\xi
		\]
	Define a set
	\[
		\Sigma = \Big\{ (t,x,y,\xi): \phi_\xi'(t,x,\xi)-y=0 \Big\}
	\]
	By \eqref{flow}, the set $\Sigma$ is contained in $\Sigma_0$.
	So the symbol $\big{(} 1-\eta(x,y) \big{)}a(t,x,\xi)$ vanishes on a neighborhood of $\Sigma$.
	By Proposition ~1.2.4 of H\"ormander ~\cite{H}, the difference between $U_a$ and $D_a$ is smoothing.
	
	At $t=0$, the determinant of the matrix $[\varphi_{\xi_i x_j}'']$ is 1.
	So if $\delta$ is small, then on the support of $a$ we can apply the implicit function theorem to the equation
	\[
		\varphi_\xi ' (t,x,\xi)-y=0
	\]
	Specifically, we can use a partition of unity to break up $a$ into a finite sum $a=\sum a_j$, so that there are functions $\psi_j(t,y,\xi)$ that are homogeneous in $\xi$ of degree zero
		such that, on the support of $a_j$, the set $\Sigma$ is given by
	\begin{equation*}
		x=\psi_j(t,y,\xi)
	\end{equation*}
	Define $b_0 \in S_{2/3, 1/3}^0(\mathbb{R}_{t,y}^3 \times \mathbb{R}_\xi^2)$ by
	\begin{equation*}
		b_0(t,y,\xi)=\sum a_j\big(t, \psi_j(t,y, \xi), \xi\big)
	\end{equation*}
	Define an operator $T_0$ with kernel
	\begin{equation*}
		\eta(x,y) \int e^{i\varphi(t,x,\xi)-iy \cdot \xi} b_0(t,y,\xi) \,d\xi
	\end{equation*}
	The difference between $U_a$ and $T_0$ is an operator with kernel
	\begin{equation*}
		\eta(x,y) \int e^{i\varphi(t,x,\xi)-iy \cdot \xi} \big{(}a(t,x,\xi)-b_0(t,y,\xi)\big{)} \,d\xi
	\end{equation*}
	The symbol $a(t,x,\xi)-b_0(t,y,\xi)$ vanishes on $\Sigma$, and the phase $\varphi(t,x,\xi)-y\cdot \xi$ is non-degenerate.
	It follows from Proposition ~1.2.5 of H\"ormander ~\cite{H} that we can write this kernel in the form
	\begin{equation*}
		\eta(x,y) \int e^{i\varphi(t,x,\xi)-iy \cdot \xi} a_0(t,x,y,\xi) \,d\xi
	\end{equation*}	
	where $a_0$ is a symbol of order $-1/3$ and type $(2/3, 1/3)$.

	Iterating this argument yields symbols $b_k(t,y,\xi)$ of order $-k/3$ and type $(2/3,1/3)$.
	These symbols are such that if $T_m$ is the operator with kernel
	\begin{equation*}
		\eta(x,y) \int e^{i\varphi(t,x,\xi)-iy \cdot \xi} \sum_{k=0}^m b_k(t,y,\xi) \,d\xi
	\end{equation*}
	then the difference between $U_a$ and $T_m$ has a kernel of the form
	\begin{equation*}
		 \eta(x,y) \int e^{i\varphi(t,x,\xi)-iy \cdot \xi} a_m(t,x,y,\xi) \,d\xi
	\end{equation*}	
	where $a_m$ is a symbol of order $-(m+1)/3$ and type $(2/3,1/3)$.
	Let $b$ be a symbol in $S_{2/3,1/3}^0(\mathbb{R}_{t,y}^{3} \times \mathbb{R}_\xi^2)$ with $b \sim \sum_{k=0}^\infty b_k$.
	Let $T_b$ be the operator with kernel
	\begin{equation*}
		\eta(x,y) \int e^{i\varphi(t,x,\xi)-iy \cdot \xi} b(t,y,\xi) \,d\xi
	\end{equation*}
	Then the difference between $U_a$ and $T_b$ is smoothing, so Lemma ~\ref{Lredsog} will follow from Lemma ~\ref{Lswitched}.
	\end{proof}
	
	The following lemma gives a suitable description of the kernel of $I_\lambda(T_b)$.
	This description is sufficiently similar to the one used in Sogge ~\cite{Sog}, so that the same argument will yield Lemma ~\ref{Lswitched}.
	
	\begin{Lemma}
	\label{kernel}
		Fix $b \in S_{2/3,1/3}^0(\mathbb{R}_{t,y}^3 \times \mathbb{R}_\xi^2)$.
		The kernel of $I_\lambda(T_b)$ is of the form
		\begin{equation}
		\label{kerneleq}
			\lambda^{1/2} e^{-i\lambda d_0(x,y)}A_\lambda(x,y) + R_\lambda(x,y)
		\end{equation}
		Here the functions $R_\lambda$ are uniformly bounded, independent of $\lambda$, and the functions $A_\lambda$ are in $C^\infty(\R^2 \times \R^2)$ satifying
		\begin{equation*}
			|\partial_x^\alpha \partial_y^\beta A_\lambda| \le C_{\alpha, \beta} \lambda^{|\beta|/3}
		\end{equation*}
		Also the functions $A_\lambda$ are supported by $x$ and $y$ in a small neighborhood of $S_r$ satisfying $\frac{1}{2}\delta \le d_0(x,y) \le \delta$.
	\end{Lemma}
	
	\begin{proof}
	The kernel of $I_\lambda(T_b)$ is
	\begin{equation*}
		\iint e^{i\varphi(t,x,\xi)-iy \cdot \xi-it\lambda} \hat{\chi}(t) \eta(x,y) b(t,y,\xi) \,d\xi dt
	\end{equation*}
	By \eqref{eikonal},
	\begin{equation*}
			\varphi(t,x,\xi) = x\cdot \xi + t p_0(x,\xi) + Q(t,x,\xi)
		\end{equation*}		
	where $Q$ is homogeneous of degree 1 in the $\xi$-variable.
	Also, for $k=0,1,2$ we have
	\begin{equation}
		\label{linapp}
			| \partial_t^k \partial_\xi^\alpha Q | \le C_{k,\alpha} t^{2-k} | \xi |^{1-|\alpha|}
		\end{equation}
	
	Let $\beta$ be a smooth function with $\beta(\xi)=1$ when $|\xi| \in [C_0^{-1}, C_0]$ and $\beta(\xi)=0$ when $|\xi| \notin [(2C_0)^{-1},2C_0]$, for some constant $C_0$.
	If $C_0$ is large and $\delta$ is small, then on the support of
	\[
		\bigg{(}1-\beta \Big( \frac{\xi}{\lambda} \Big) \bigg{)}\hat{\chi}(t) \eta(x,y) b(t,y,\xi)
	\]
	we have
	\begin{equation*}
			\Big{|} \frac{\partial}{\partial t}\Big{(} \varphi(t,x,\xi)-y\cdot\xi-t\lambda \Big{)} \Big{|} \; \gtrsim \; p_0(x,\xi)+\lambda \; \gtrsim \; 1+|\xi|
		\end{equation*}
	So for any positive integer $N$,
	\begin{equation*}
			\int e^{i\varphi(t,x,\xi)-iy \cdot \xi-it\lambda} \bigg{(}1-\beta\Big{(}\frac{\xi}{\lambda}\Big{)}\bigg{)}\hat{\chi}(t) \eta(x,y) b(t,y,\xi) \,dt \le C_N (1+|\xi|)^{-N}
		\end{equation*}
	This implies that the difference between the kernel of $I_\lambda(T_b)$ and
	\begin{equation}
		\label{localized}
			\iint e^{i\varphi(t,x,\xi)-iy \cdot \xi-it\lambda} \beta\Big{(}\frac{\xi}{\lambda}\Big{)}\hat{\chi}(t) \eta(x,y) b(t,y,\xi) \,d\xi dt
		\end{equation}
	is bounded uniformly in $\lambda$.
	
	Now it suffices to show that \eqref{localized} can be written as in \eqref{kerneleq}.
	After changing variables \eqref{localized} becomes
	\[
		\label{cov}
			\lambda^2 \iint e^{i\lambda \Phi(t,x,y,\xi)} p_\lambda(t,x,y,\xi) \,d\xi dt
		\]
	where the phase is
	\[
			\Phi(t,x,y,\xi)=\varphi(t,x,\xi)-y \cdot \xi - t
		\]
	and the amplitude is
	\[
			p_\lambda(t,x,y,\xi)=\beta(\xi)\hat{\chi}(t) \eta(x,y) b(t,y,\lambda \xi)
		\]
	Here $p_\lambda$ is smooth and compactly supported with
	\[
		\label{pbounds}
			| \partial_t^k \partial_x^\alpha \partial_y^\beta \partial_\xi^\gamma p_\lambda | \le C_{k,\alpha,\beta,\gamma} \lambda^{(k+|\beta|+|\gamma|)/3}
		\]

	To apply stationary phase, the Hessian of $\Phi$, with respect to the $(t,\xi)$-variables, must be non-degenerate on the support of $p_\lambda$.
	First note that its determinant is homogeneous of degree $-1$ in the $\xi$-variable.
	We have
	\[
			\Phi(t,x,y,\xi)=(x-y)\cdot \xi + t p_0(x,\xi)-t+Q(t,x,\xi)
		\]
	We can compute explicitly the Hessian of
	\[
			(x-y)\cdot \xi + t p_0(x,\xi)-t
		\]
	with respect to the $(t,\xi)$-variables.
	Its determinant is
	\[
			-\frac{t}{p_0(x,\xi)}\det g^{jk}
		\]
	Now it follows from \eqref{linapp} that the determinant of the Hessian of $\Phi$, with respect to the $(t,\xi)$-variables, is
	\[
			-\frac{t}{p_0(x,\xi)}\det g^{jk} + t^2q(t,x,y,\xi)
		\]
	where $q$ is a smooth function, homogeneous of degree $-1$ in the $\xi$-variable.
	So if $\delta$ is small, then the Hessian of $\Phi$, with respect to the $(t,\xi)$-variables, is non-degenerate on the support of $p_\lambda$.
	
	The critical points of $\Phi$, with respect to the $(t,\xi)$-variables, are the solutions of
	\[
		\label{crit}
			\varphi_\xi '(t,x,\xi)=y \quad \text{and} \quad \varphi_t'(t,x,\xi)=1
		\]
	We can use the implicit function theorem at any critical point.
	By using a partition of unity and abusing notation, we can assume that
		there are smooth functions $t(x,y)$ and $\xi(x,y)$, such that if $\delta$ is small, then on the support of $p_\lambda$, the critical points are given by
	\[
		\big{(}t(x,y),x,y,\xi(x,y)\big{)}
	\]
	Because of \eqref{flow}, we have $t(x,y)=d_0(x,y)$.
	Applying Euler's homogeneity relation $\varphi=\varphi_\xi ' \cdot \xi$ yields
	\[
			\Phi\big{(}t(x,y),x,y,\xi(x,y)\big{)}=-t(x,y)=-d_0(x,y)
		\]
	So Lemma ~\ref{kernel} follows from the following stationary phase lemma.
	\end{proof}
	
	\begin{Lemma}
	\label{staphase}
		Consider the oscillatory integrals
		\[
				J_\lambda(x,y)=\int_{\mathbb{R}^3} e^{i\lambda \Psi(x,y,z)} q_\lambda(x,y,z) \, dz
			\]
		where $\Psi$ is a smooth real function and the amplitudes $q_\lambda$ are smooth with fixed compact support and satisfy
		\[
				| \partial_x^\alpha \partial_y^\beta \partial_z^\gamma q_\lambda | \le C_{\alpha,\beta,\gamma} \lambda^{(|\beta|+|\gamma|)/3}
			\]
		Assume that on the support of the symbols $q_\lambda$, the Hessian of $\Psi$ with respect to the $z$-variable is non-degenerate and
			the solutions of $\Psi_z'(x,y,z)=0$ are given by $(x,y,z(x,y))$ where $z(x,y)$ is a smooth function.
		Then
		\[
				\Big{|} \partial_x^\alpha \partial_y^\beta \Big{(}e^{-i\lambda \Psi(x,y,z(x,y))} J_\lambda(x,y) \Big{)}\Big{|} \le C_{\alpha,\beta} \lambda^{-3/2+|\beta|/3}
			\]
		
	\end{Lemma}
	
	This lemma is similar to Corollary ~1.1.8 in Sogge ~\cite{SogBook}, which dealt with symbols $q_\lambda$ with derivatives bounded independent of $\lambda$.
	Essentially the same proof as in Sogge ~\cite{SogBook} yields Lemma ~\ref{staphase}, and then Lemma ~\ref{kernel} follows.
	We can now obtain Lemma ~\ref{Lswitched} by using the argument in Sogge ~\cite{Sog}.
	
\subsection*{Argument from Sogge ~\cite{Sog}}
	
	To finish the proof of Lemma ~\ref{Lswitched} it suffices to show that for any $\varepsilon > 0$ there is a constant $C_\varepsilon$ such that
	\begin{multline}
		\label{sogstart}
			\int_{S_r} \Big| \lambda^{1/2} \int e^{-i\lambda d_0(x,y)} A_\lambda(x,y)f(y)\,dy \Big|^2 | g(x) |^2 \,dx 
			\\
			\le \varepsilon \lambda^{1/4} \| f \|_{L^2(\mathbb{R}^2)}^2 \| g \|_{L^4(\mathbb{R}^2)}^2
			+ C_\varepsilon \lambda^{1/2} \| f \|_{L^2(\mathbb{R}^2)}^2 \sup_{\gamma \in \Pi_0} \| g \|_{L^2(\mathcal{T}_{\lambda}(\gamma))}^2
		\end{multline}
	
	By using a partition of unity and abusing notation, we can assume there are points $x_0$ and $y_0$ with $x_0$ in $S_r$ and $\delta/2 \le d_0(x_0,y_0) \le \delta$
		such that $A_\lambda$ is supported by $x$ in a small neighborhood $\mathcal{N}_x$ of $x_0$ and $y$ in a small neighborhood $\mathcal{N}_y$ of $y_0$.
	In particular, we assume that $\mathcal{N}_x$ and $\mathcal{N}_y$ are, respectively,
		contained in $B(x_0,\delta/5)$ and $B(y_0,\delta/5)$, the geodesic balls of radius $\delta/5$ around $x_0$ and $y_0$, respectively.
	
	We will work in Fermi normal coordinates $(\sigma,\tau)_F$ about $\gamma_0$, the geodesic going through $x_0$ which is orthogonal to the geodesic connecting $x_0$ and $y_0$.
	These coordinates are well defined on $B(x_0,2\delta)$ if $\delta$ is small enough.
	These coordinates are such that $\gamma_0$ is given by a vertical line parallel to the $\tau$-axis,
		and the geodesics which intersect $\gamma_0$ orthogonally are given by horizontal lines parallel to the $\sigma$-axis.
	Also $x_0$ lies on the negative $\sigma$-axis and $y_0$ on the positive $\sigma$-axis.
	Now it suffices to prove
	\begin{multline*}
			\int \Big( \int_{S_r} \Big| \lambda^{1/2} \int e^{-i\lambda d_0(x,(\sigma,\tau)_F)} A_\lambda\big(x,(\sigma,\tau)_F\big)f(\sigma,\tau)\,d\tau \Big|^2 | g(x) |^2 \,dx \Big) \,d\sigma
			\\
			\le \varepsilon \lambda^{1/4} \| f \|_{L^2(\mathbb{R}^2)}^2 \| g \|_{L^4(\mathbb{R}^2)}^2
				+ C_\varepsilon \lambda^{1/2} \| f \|_{L^2(\mathbb{R}^2)}^2 \sup_{\gamma \in \Pi_0} \| g \|_{L^2(\mathcal{T}_{\lambda}(\gamma))}^2
		\end{multline*}
	This will follow if we show
	\begin{multline}
		\label{fermi}
			\int_{S_r} \Big| \lambda^{1/2} \int e^{-i\lambda d_0(x,(\sigma,\tau)_F)} A_\lambda\big(x,(\sigma,\tau)_F\big)h(\tau)\,d\tau \Big|^2 | g(x) |^2 \,dx
			\\
			\le \varepsilon \lambda^{1/4} \| h \|_{L^2(\mathbb{R})}^2 \| g \|_{L^4(\mathbb{R}^2)}^2
				+ C_\varepsilon \lambda^{1/2} \| h \|_{L^2(\mathbb{R})}^2 \sup_{\gamma \in \Pi_0} \| g \|_{L^2(\mathcal{T}_{\lambda}(\gamma))}^2
		\end{multline}
	where $C_\varepsilon$ is independent of $\sigma$.
	To simplify the notation, we will only prove this for a fixed value of $\sigma$, which we may take to be zero by relabeling the coordinates.
	The argument will also yield the uniformity in $\sigma$.
	Note that after relabeling, we can assume that the point $(0,0)_F$ is in $\mathcal{N}_y$.
	Then $x_0=(-\sigma_0,0)_F$ where $\sigma_0>\delta/4$.
	
	We take a smooth bump function $\eta \in C_0^\infty(\mathbb{R})$ supported in $[-1,1]$ and satisfying $\sum_{j\in \mathbb{Z}} \eta(\tau-j)=1$.
	Define
	\[
			\eta_{\lambda,j}(\tau)=\eta(\lambda^{1/2}\tau-j)
		\]
	Let
	\[
			z_j=z_j(\lambda,x,h)=\lambda^{1/2} \int e^{-i\lambda d_0(x,(0,\tau)_F)} \eta_{\lambda,j}(\tau)A_\lambda \big(x, (0,\tau)_F\big) h(\tau) \,d\tau
		\]
	Then for $N=1,2,3,\dots$,
	\begin{multline*}
			\Big| \sum_{j,k\in \mathbb{Z}} z_j z_k \Big| \le \Big| \sum_{|j-k| > N} z_j z_k \Big| + \Big| \sum_{|j-k|\le N} z_j z_k \Big|
			\\
			\le \Big| \sum_{|j-k| > N} z_j z_k \Big| + \sum_{|j-k|\le N} \frac{1}{2}\Big( |z_j|^2 + |z_k|^2 \Big)
			\\
			\le \Big| \sum_{|j-k| > N} z_j z_k \Big| + (2N+1)\sum_{j\in\mathbb{Z}} |z_j|^2
		\end{multline*}
	This means that
	\begin{multline}
		\label{fork}
			\Big| \lambda^{1/2} \int e^{-i\lambda d_0(x,(0,\tau)_F)} A_\lambda \big(x, (0,\tau)_F\big) h(\tau) \,d\tau \Big|^2
			\\
			\le \Big| \lambda \iint e^{-i\lambda [d_0(x,(0,\tau)_F)+d_0(x,(0,\tau')_F)]} B_{N, \lambda} (x,\tau,\tau') h(\tau) h(\tau') \, d\tau d\tau' \Big|
			\\
			+ (2N+1) \sum_{j \in \mathbb{Z}} \lambda \Big| \int e^{-i\lambda d_0(x,(0,\tau)_F)} \eta_{\lambda,j}(\tau)A_\lambda \big(x, (0,\tau)_F\big) h(\tau) \,d\tau \Big|^2
		\end{multline}
	where
	\[
			B_{N,\lambda}(x,\tau,\tau')= \sum_{|j-k| > N} \eta_{\lambda,j}(\tau) A_\lambda\big(x,(0,\tau)_F\big) \eta_{\lambda,k}(\tau') A_\lambda\big(x,(0,\tau')_F\big)
		\]
	We will prove
	\begin{multline}
		\label{bigdiff}
			\Big\| \lambda \iint e^{-i\lambda[d_0(x,(0,\tau)_F)+d_0(x,(0,\tau')_F)]}B_{N,\lambda}(x,\tau,\tau') h(\tau) h(\tau') \,d\tau d\tau' \Big\|_{L_x^2(S_r)}
			\\
			\lesssim \lambda^{1/4} N^{-1/2} \| h \|_{L^2(\mathbb{R})}^2
		\end{multline}
	and
	\begin{multline}
		\label{lildiff}
			\int_{S_r} \lambda \, \Big| \int e^{-i\lambda d_0(x,(0,\tau)_F)} \eta_{\lambda,j}(\tau) A_\lambda\big(x,(0,\tau)_F\big) H(\tau) \,d\tau \Big|^2 |g(x)|^2 \,dx
			\\
			\lesssim \lambda^{1/2} \| H \|_{L^2(\mathbb{R})}^2 \sup_{\gamma \in \Pi_0} \| g \|_{L^2(\mathcal{T}_{\lambda}(\gamma))}^2
		\end{multline}
	Let $\chi_{\lambda,j}$ be the characteristic function of supp $\eta_{\lambda,j}$.
	Then \eqref{lildiff} will yield
	\begin{multline}
		\label{lildiff'}
			 \sum_{j \in \mathbb{Z}} \int_{S_r} \lambda \Big| \int e^{-i\lambda d_0(x,(0,\tau)_F)} \eta_{\lambda,j}(\tau)A_\lambda \big(x, (0,\tau)_F\big) h(\tau) \,d\tau \Big|^2 \,dx
			 \\
			 \lesssim \sum_{j \in \mathbb{Z}} \lambda^{1/2} \| h \chi_{\lambda,j} \|_{L^2(\mathbb{R})}^2 \sup_{\gamma \in \Pi_0} \| g \|_{L^2(\mathcal{T}_{\lambda}(\gamma))}^2
			 \\
			 \lesssim \lambda^{1/2} \| h \|_{L^2(\mathbb{R})}^2 \sup_{\gamma \in \Pi_0} \| g \|_{L^2(\mathcal{T}_{\lambda}(\gamma))}^2
		\end{multline}
	Then \eqref{fork}, \eqref{bigdiff}, and \eqref{lildiff'} will yield \eqref{fermi}.
	So it remains to prove \eqref{bigdiff} and \eqref{lildiff}.
	
	The inequality \eqref{bigdiff} will be a consequence of the following lemma.
	\begin{Lemma}
		\label{Lcsj}
			Let $B_\lambda(x,\tau,\tau')$ be a smooth function over $\mathbb{R}^4$ with $|\partial_x^\alpha B_\lambda| \le C_\alpha$
				and assume $B_\lambda$ vanishes unless $|x| \le \delta_0$ and $|\tau - \tau'| \le \delta_0$.
			Assume that $\mu(x,\tau)$ is a real smooth function over $\mathbb{R}^3$ satisfying the Carleson-Sj\"olin condition on the support of the amplitudes $B_\lambda$, that is
			\[
					\det
						\begin{pmatrix}
								\mu_{x_1 \tau}'' & \mu_{x_2 \tau}'' \\
								\mu_{x_1 \tau \tau}''' & \mu_{x_2 \tau \tau}'''
							\end{pmatrix}
						\neq 0
				\]
			If $\delta_0>0$ is sufficiently small, then
			\begin{multline}
				\label{csjconc}
					\Big\| \iint_{|\tau-\tau'| \ge N\lambda^{-1/2}} e^{i\lambda[\mu(x,\tau)+\mu(x,\tau')]} B_\lambda(x,\tau,\tau') F(\tau,\tau') \,d\tau d\tau' \Big\|_{L^2_x(S_r)}^2
					\\
					\lesssim \lambda^{-3/2} N^{-1} \| F \|_{L^2(\mathbb{R}^2)}^2
				\end{multline}
			Moreover, if the $C_\alpha$ are fixed and $\delta_0$ is sufficiently small, this estimate is uniform over all functions $B_\lambda$ which satisfy the hypotheses.
		\end{Lemma}
	
	It is well known that the function $\mu(x,\tau)=-d_0(x,(0,\tau)_F)$ satisfies the Carleson-Sj\"olin condition.
	So Lemma ~\ref{Lcsj} will imply \eqref{bigdiff}.
	
	\begin{proof}[Proof of Lemma ~\ref{Lcsj}]
	Let $\Upsilon(x,\tau,\tau')=\mu(x,\tau)+\mu(x,\tau')$.
	Then the determinant of the mixed Hessian of $\Upsilon$ satisfies
	\[
			\Big| \det \Big( \frac{\partial^2 \Upsilon}{\partial x \partial(\tau,\tau')} \Big) (x,\tau,\tau') \Big| = \mu_{x_1 \tau}''(x,\tau)\mu_{x_2 \tau'}''(x, \tau')-\mu_{x_1 \tau'}''(x,\tau')\mu_{x_2 \tau}''(x, \tau)		\]
	By the Carleson-Sj\"olin condition, the $\tau'$ derivative of this function is nonzero on the diagonal $\tau=\tau'$.
	This implies that
	\[
			\Big| \det \Big( \frac{\partial^2 \Upsilon}{\partial x \partial(\tau,\tau')} \Big) \Big| \ge c |\tau-\tau'|
		\]
	for some $c>0$ on the support of the amplitudes $B_\lambda$, if $\delta_0$ is small.
	We use the change of variables
	\[
			u = (\tau-\tau', \tau+\tau')
		\]
	Since $|du/d(\tau,\tau')|=2$, we obtain
	\[
			\Big| \det \Big( \frac{\partial^2 \Upsilon}{\partial x \partial u} \Big) \Big| \ge c |u_1|
		\]
	Now $\Upsilon$ is an even function in the $u_1$-variable, so it is a smooth function of $u_1^2$.
	We can make another change of variables
	\[
			v = (\frac{1}{2}u_1^2,u_2).
		\]
	Then $|dv/du|=|u_1|$, so
	\[
			\Big| \det \Big( \frac{\partial^2 \Upsilon}{\partial x \partial v} \Big) \Big| \ge c
		\]
	This implies that if $v$ and $\tilde{v}$ are close then
	\[
			\Big| \nabla_x[\Upsilon(x,v)-\Upsilon(x,\tilde{v})] \Big| \ge c' |v-\tilde{v}|
		\]
	for some $c'>0$.
	Since $\Upsilon$ is smooth as a function of $x$ and $v$,
	\[
			\Big| \partial_x^\alpha [\Upsilon(x,v)-\Upsilon(x,\tilde{v})] \Big| \le C_\alpha' |v-\tilde{v}|
		\]
		
	Now if we define
	\[
			K_\lambda(v,\tilde{v}) = \int_{S_r} B_\lambda(x,\tau,\tau') \overline{B_\lambda(x,\tilde{\tau},\tilde{\tau}')} e^{i\lambda[\Upsilon(x,v)-\Upsilon(x,\tilde{v})]} \,dx
		\]
	then for $j=1,2,3,\ldots$, integrating by parts yields
	\begin{equation}
		\label{diag0}
			| K_\lambda(v,\tilde{v}) | \le C_j (1+\lambda |v-\tilde{v}|)^{-2j}
		\end{equation}
	For $a, b \ge0$,
	\[
			(1+2a)(1+b) \le 2\Big( 1+(a^2+b^2)^{1/2} \Big)^2
		\]
	If we set $a=\lambda |v_1-\tilde{v}_1|$ and $b=\lambda |v_2-\tilde{v}_2|$, then \eqref{diag0} becomes
	\begin{equation}
		\label{diag}
			| K_\lambda(v,\tilde{v}) | \le C_j' (1+\lambda |(u_1^2-\tilde{u}_1^2|)^{-j} (1+\lambda |u_2-\tilde{u}_2|)^{-j}
		\end{equation}
	Let $E_{N,\lambda}$ be the characteristic function of the set
	\[
			\{ (u, \tilde{u}) \in \mathbb{R}^4 : |u_1|,|\tilde{u}_1| \ge N \lambda^{-1/2} \}
		\]
	Then the left side of \eqref{csjconc} equals
	\[
			\iint E_{N,\lambda}(u, \tilde{u}) K_\lambda(u,\tilde{u}) F(u) \overline{F(\tilde{u})} \,du d\tilde{u}
		\]
	By H\"older's inequality, it remains to prove that
	\[
			\Big\| \int E_{N,\lambda}(u, \tilde{u}) K_\lambda(u,\tilde{u}) F(u) \,du \Big\|_{L^2_{\tilde{u}}(\mathbb{R}^2)} \lesssim \lambda^{-3/2}N^{-1} \| F \|_{L^2(\mathbb{R}^2)}
		\]
	This will follow from Young's inequality, if we show that
	\[
			\sup_{\tilde{u}} \int_{|u_1| \ge N\lambda^{-1/2}} |K_\lambda(u,\tilde{u})| \,du \lesssim \lambda^{-3/2}N^{-1}
		\]
	and
	\[
			\sup_{u} \int_{|\tilde{u}_1| \ge N\lambda^{-1/2}} |K_\lambda(u,\tilde{u})| \,d\tilde{u} \lesssim \lambda^{-3/2}N^{-1}
		\]
	Because of \eqref{diag}, both of these inequalities will follow if we check that
	\begin{equation}
		\label{cal}
			\sup_{c_1,c_2\in\mathbb{R}} \int_{w_1\ge N\lambda^{-1/2}} (1+\lambda |w_1^2-c_1|)^{-2} (1+\lambda |w_2-c_2|)^{-2} \,dw \lesssim \lambda^{-3/2}N^{-1}
		\end{equation}
	By changing variables,
	\begin{equation}
		\label{calc}
			\sup_{c_2\in \mathbb{R}} \int (1+\lambda |w_2-c_2|)^{-2} \,dw_2 = \lambda^{-1} \int (1+ |\tilde{w}_2|)^{-2} \,d\tilde{w}_2 \lesssim \lambda^{-1}
		\end{equation}
	If we set $z=w_1^2$, then $dw_1=\frac{1}{2}z^{-1/2}dz$, so we also have
	\begin{multline}
		\label{calcu}
			\sup_{c_1\in \mathbb{R}} \int _{w_1\ge N\lambda^{-1/2}} (1+\lambda |w_1^2-c_1|)^{-2} \,dw_1
			\\
			=\frac{1}{2} \sup_{c_1\in \mathbb{R}} \int _{z\ge N^2\lambda^{-1}} (1+\lambda |z-c_1|)^{-2} z^{-1/2} \,dz
			\\
			\le \lambda^{1/2} N^{-1} \sup_{c_1\in \mathbb{R}} \int (1+\lambda |z-c_1|)^{-2} \,dz
			\\
			\le \lambda^{-1/2} N^{-1} \int (1+ |\tilde{z}|)^{-2} \,d\tilde{z} \lesssim \lambda^{-1/2}N^{-1}
		\end{multline}
	Now \eqref{calc} and \eqref{calcu} yield \eqref{cal}, completing the proof of Lemma \ref{Lcsj}.
	\end{proof}
	
	So we have proven \eqref{bigdiff}, and it remains to show \eqref{lildiff}.
	To simplify the notation, we will only prove this for $j=0$.
	The argument will also show that \eqref{lildiff} holds for all $j$ in $\mathbb{Z}$, uniformly.
	
	Let $p=(0,0)_F$.
	Let $T$ be the tangent plane at $p$.
	The exponential map is a diffeomorphism from a ball of radius $2\delta$ in $T$ to $B\big(p,2\delta\big)$ if $\delta$ is small.
	Let $\kappa$ be the inverse function.
	We will identify $T$ with $\mathbb{R}^2$ in such a way that the Riemannian metric on $T$ agrees with the Euclidean metric on $\mathbb{R}^2$.
	We can make this identification in such a way that $\exp_p(\sigma,0)=(\sigma,0)_F$ for all $\sigma$.
	Let $\kappa_1$ and $\kappa_2$ denote the component functions of $\kappa$, so that $\kappa=(\kappa_1,\kappa_2)$.
	The inequality \eqref{lildiff} will be a consequence of the following lemma.
	
	\begin{Lemma}
		\label{Lgauss}
			Let $\psi(x,\tau)=-d_0\big(x,(0,\tau)_F\big)$ and let $\rho_\lambda$ be functions in $C_0^\infty(\mathbb{R}^3)$ satisfying
			\begin{equation}
				\label{gc1}
					| \partial_\tau^m \rho_\lambda(x,\tau) | \le C_m \lambda^{m/2}
				\end{equation}
			and
			\begin{equation}
				\label{gc2}
					\text{supp }\rho_\lambda \subset \Big\{ (x,\tau) : |\tau| \le \lambda^{-1/2}, x \in \mathcal{N}_x,(0,\tau)_F \in \mathcal{N}_y \Big\}
				\end{equation}
			Assume $q_k$ are points in $\mathcal{N}_x$ satisfying
			\begin{equation}
				\label{sinecon}
					\Big| \frac{\kappa_2(q_k)}{|\kappa(q_k)|} - \frac{\kappa_2(q_\ell)}{|\kappa(q_\ell)|} \Big| \ge c \lambda^{-1/2}|k-\ell |
				\end{equation}
			with $c>0$, when $|k-\ell| \ge 2$.
			If $\mathcal{N}_x$ is sufficiently small, then
			\begin{equation}
				\label{gconc}
					\lambda^{1/2} \int \Big| \sum_k e^{i\lambda \psi(q_k,\tau)} \rho_\lambda(q_k,\tau) p_k \Big|^2 \,d\tau \lesssim \sum |p_k|^2
				\end{equation}
			This estimate is uniform over different choices of the points $q_k$.
		\end{Lemma}
	
	To see that Lemma ~\ref{Lgauss} implies \eqref{lildiff}, let $\kappa_r(x)$ and $\kappa_\theta(x)$ be the polar coordinates of $\kappa(x)$ with $\kappa_\theta(x)$ in $[0,2\pi)$.
	These functions are well defined and smooth on $\mathcal{N}_x$.
	Define
	\[
			\rho_\lambda(x,\tau)=\eta_{\lambda,0}(\tau) A_\lambda\big(x,(0,\tau)_F\big)
		\]
	Then \eqref{gc1} and \eqref{gc2} hold.
	Define the sets
	\[
			V_k=\Big\{ x \in \mathcal{N}_x: \lambda^{-1/2}k \le \kappa_\theta(x) < \lambda^{-1/2}(k+1) \Big\}
		\]
	We have
	\begin{multline*}
			\int_{S_r} \lambda \, \Big| \int e^{-i\lambda d_0(x,(0,\tau)_F)} \eta_{\lambda,0}(\tau) A_\lambda\big(x,(0,\tau)_F\big) H(\tau) \,d\tau \Big|^2 |g(x)|^2 \,dx
			\\
			\le \sum_k \lambda \Big\| \int e^{i\lambda \psi(x,\tau)} \rho_\lambda(x,\tau) H(\tau) \,d\tau \Big\|_{L_x^\infty(V_k)}^2 \| g \|_{L^2(V_k)}^2
			\\
			\le \sup_\ell \| g \|_{L^2(V_\ell)}^2 \sum_k \lambda \Big\| \int e^{i\lambda \psi(x,\tau)} \rho_\lambda(x,\tau) H(\tau) \,d\tau \Big\|_{L_x^\infty(V_k)}^2
		\end{multline*}
	If $\mathcal{N}_x$ is small, then each $V_\ell$ is contained in $\mathcal{T}_\lambda(\gamma_\ell)$ for some $\gamma_\ell \in \Pi_0$.
	In fact, each $\gamma_\ell$ can be chosen to go through $p$.
	This yields
	\[
			\sup_\ell \| g \|_{L^2(V_\ell)}^2 \le \sup_{\gamma\in\Pi_0} \| g \|_{L^2(\mathcal{T}_\lambda(\gamma))}^2
		\]
	Now to prove \eqref{lildiff}, it remains to show that
	\[
			\sum_k \lambda^{1/2} \Big\| \int e^{i\lambda \psi(x,\tau)} \rho_\lambda(x,\tau) H(\tau) \,d\tau \Big\|_{L_x^\infty(V_k)}^2 \lesssim \| H \|_{L^2(\mathbb{R})}^2
		\]
	It suffices to check that for any choice of points $q_k$ in $V_k$,
	\begin{equation*}
		\label{predual}
			\sum_k \lambda^{1/2} \Big| \int e^{i\lambda \psi(q_k,\tau)} \rho_\lambda(q_k,\tau) H(\tau) \,d\tau \Big|^2 \lesssim \| H \|_{L^2(\mathbb{R})}^2
		\end{equation*}
	and that this holds uniformly over different choices of $q_k$.
	By duality, this inequality is equivalent to \eqref{gconc}.
	To apply Lemma ~\ref{Lgauss}, we still need to check that any choice of points $q_k$ in $S_k$ satisfies \eqref{sinecon}.
	If $\mathcal{N}_x$ and $\mathcal{N}_y$ are sufficiently small, then $\kappa_\theta(\mathcal{N}_x)$ is contained in $[2\pi/3, 4\pi/3]$.
	When $|j-k| \ge 2$, we then have
	\begin{multline*}
			\Big| \frac{\kappa_2(q_j)}{|\kappa(q_j)|} - \frac{\kappa_2(q_k)}{|\kappa(q_k)|} \Big| = \Big| \sin\big(\kappa_\theta(q_j)\big)- \sin\big(\kappa_\theta(q_k)\big) \Big|
			\\
			\ge \frac{1}{2} \Big| \kappa_\theta(q_j)-\kappa_\theta(q_k) \Big| \ge \frac{1}{4} \lambda^{-1/2}|j-k|
		\end{multline*}
	This is \eqref{sinecon}, so Lemma ~\ref{Lgauss} will imply \eqref{lildiff}.
	
	\begin{proof}[Proof of Lemma \ref{Lgauss}]
	We can write
	\[
			\psi(x,\tau) = \psi(x,0) + \tau \partial_\tau\psi(x,0)+r(x,\tau)
		\]
	where
	\[
			| r(\tau,x) | \le C_0 |\tau|^2 \quad | \partial_\tau r(\tau,x) | \le C_1 |\tau|
		\]
	and for $m=2,3,\ldots$
	\[
			| \partial_\tau^m r(\tau,x) | \le C_m
		\]
	
	Fix $x$ in $\mathcal{N}_x$ and let $\Theta$ be the geodesic sphere of radius $|\kappa(x)|$ around $x$.
	By Gauss' lemma, $\kappa(x)$ is normal to $\kappa(\Theta)$.
	Define a function $G$ from $\mathbb{R}^2$ to $\mathbb{R}$ by
	\[
		G(u)=-d_0(x,\exp_p(u))
	\]
	Then $\kappa(\Theta)$ is a level set of $G$, so $\nabla G(0)$ is normal to $\kappa(\Theta)$.
	That is, $\nabla G(0)$ is a multiple of $\kappa(x)$.
	Define a curve $c$ in $T$ by $c(t)=t\kappa(x)$.
	Then $G(c(t))=(t-1)|\kappa(x)|$ for $t$ near $0$, so $\nabla G(0) \cdot \kappa(x)=|\kappa(x)|$.
	Since $\nabla G(0)$ is a multiple of $\kappa(x)$, this implies that
	\[
		\nabla G(0) = \frac{\kappa(x)}{|\kappa(x)|}
	\]
	This yields
	\[
		\partial_\tau \psi(x,0) = \nu \cdot \frac{\kappa(x)}{|\kappa(x)|}
	\]
	where
	\[
		\nu = \partial_\tau \kappa \big((0,\tau)_F \big) \Big|_{\tau=0}
	\]
	That is, $\nu$ is the pushforward under $\kappa$ of $\partial/\partial\tau$ at $p$.
	It must be transverse to the pushforward under $\kappa$ of $\partial/\partial\sigma$ at $p$, whose second component is zero.
	So the second component of $\nu$ is nonzero.
	By \eqref{sinecon},
	\[
			\Big| \partial_\tau \psi(q_k,0) - \partial_\tau \psi(q_\ell,0) \Big| \ge c' \lambda^{-1/2} |j-k|
		\]
	for some $c' > 0$ when $|k-\ell | \ge 2$.
	
	Now define
	\[
			P_\lambda(q_k,q_\ell,\tau)=\rho_\lambda(q_k,\tau) \overline{\rho_\lambda(q_\ell,\tau)} e^{i\lambda[\psi(q_k,0)+r(q_k,\tau)]} e^{-i\lambda[\psi(q_\ell,0)+r(\ell,\tau)]}
		\]
	Then $P_\lambda(q_k,q_\ell,\tau)$ vanishes when $|\tau| \ge \lambda^{-1/2}$ and satisfies
	\[
			\Big| \partial_\tau^m P_\lambda(q_k,q_\ell,\tau) \Big| \le C_m \lambda^{m/2}
		\]
	The left side of \eqref{gconc} is equal to
	\[
			\lambda^{1/2} \sum_{k,\ell} p_k \overline{p_\ell} \Big( \int e^{i\tau\lambda[ \partial_\tau \psi(q_k,0)-\partial_\tau \psi(q_\ell,0)]} P_\lambda(q_k,q_\ell,\tau) \, d\tau \Big)
		\]
	We integrate by parts twice to control this by
	\[
			\sum_{k,\ell} |p_k p_\ell | (1+|k-\ell |)^{-2} \lesssim \sum_{k,\ell} (|p_k|^2+|p_\ell|^2)(1+|k-\ell|)^{-2} \lesssim \sum_k |p_k|^2
		\]
	This completes the proof of Lemma ~\ref{Lgauss}, and now Theorem ~\ref{thlink} follows.
	\end{proof}
	
\section{End of Proof of Theorem 1.1}

	To complete the proof of Theorem ~\ref{thbgt}, it remains to prove Lemma ~\ref{Lredbgt}.
	This will be a consequence of the following variant.
	To state it, recall that $\eta(x,y)$ is in $C_0^\infty(\R^2 \times \R^2)$ and is supported by $x$ and $y$ in a small neighborhood of $S_r$ satisfying $\frac{1}{2}\delta \le d_0(x,y) \le \delta$.
	Also $\eta(x,y)=1$ when $x$ is in a small neighborhood of $S_r$ and $d_0(x,y)$ is in an open neighborhood of the support of $\hat{\chi}$.

	\begin{Lemma}
	\label{Lunswitched}
		Fix $a \in S_{2/3,1/3}^0(\mathbb{R}_{t,x}^3 \times \mathbb{R}_\xi^2)$.
		Define an operator $D_a$ by
		\begin{equation*}
			D_a f = \iint e^{i\varphi(t,x,\xi)-iy\cdot\xi} \eta(x,y) a(t,x,\xi) f(y)\, d\xi dy
		\end{equation*}
		For any smooth curve $\Gamma$ in $S_r$ of unit length, and for $f$ with fixed compact support,
		\[
				\| I_\lambda(D_a)f \|_{L^4(\Gamma)} \lesssim \lambda^{1/4} \| f \|_{L^2(\mathbb{R}^2)}
			\]
	\end{Lemma}
	
	Using Lemma \ref{Lunswitched}, we can now prove Lemma \ref{Lredbgt}.
	\begin{proof}[Proof of Lemma \ref{Lredbgt}]
	Fix a symbol $a \in S_{2/3, 1/3}^0(\mathbb{R}_{t,x}^3 \times \mathbb{R}_\xi^2)$.
	We may assume that $\big(1-\eta(x,y)\big)a(t,x,\xi)$ vanishes on a neighborhood of the set 
	\[
			\Sigma_0=\{ (t,x,y,\xi): t=d_0(x,y) \}
		\]
	We can make this assumption because $I_\lambda(U_a)$ only depends on $t$ in the support of $\hat{\chi}$.
	The kernel of $U_a$ is
	\begin{equation*}
		 \int e^{i\varphi(t,x,\xi)-iy \cdot \xi} a(t,x,\xi) \,d\xi
	\end{equation*}
	Define a set
	\[
		\Sigma=\Big\{ (t,x,y,\xi) : \phi_\xi'(t,x,\xi)-y=0 \Big\}
	\]
	Define an operator $D_a$ with kernel
	\[
			\int e^{i\varphi(t,x,\xi)-iy \cdot \xi} \eta(x,y) a(t,x,\xi) \,d\xi
		\]
	By \eqref{flow}, the set $\Sigma$ is contained in $\Sigma_0$.
	So the symbol $\big{(} 1-\eta(x,y) \big{)}a(t,x,\xi)$ vanishes on a neighborhood of $\Sigma$.
	By Proposition ~1.2.4 of H\"ormander ~\cite{H}, the difference between $U_a$ and $D_a$ is smoothing, so it suffices to show that $I_\lambda(D_a)$ satisfies \eqref{coord1}.
	Any broken geodesic $\gamma$ in $S_r$ can be broken up into a fixed finite number of segments which are smooth curves, so this will follow from Lemma ~\ref{Lunswitched}.
	\end{proof}
	
	The next lemma will give a suitable description of the kernel of $I_\lambda(D_a)$.
	This description is sufficiently similar to the one used in Burq-G\'erard-Tzvetkov ~\cite{BGT}, so that the same argument will yield Lemma ~\ref{Lunswitched}.
	
	\begin{Lemma}
	\label{unkernel}
		Fix $a \in S_{2/3,1/3}^0(\mathbb{R}_{t,x}^3 \times \mathbb{R}_\xi^2)$.
		The kernel of $I_\lambda(D_a)$ is of the form
		\begin{equation}
		\label{unkerneleq}
			\lambda^{1/2} e^{-i\lambda d_0(x,y)}A_\lambda(x,y) + R_\lambda(x,y)
		\end{equation}
		where $R_\lambda$ is uniformly bounded in $\lambda$ and $A_\lambda$ is in $C^\infty(\mathbb{R}^2 \times \mathbb{R}^2)$ and satisfies
		\begin{equation*}
			|\partial_x^\alpha \partial_y^\beta A_\lambda| \le C_{\alpha, \beta} \lambda^{|\alpha|/3}
		\end{equation*}
		Also $A_\lambda$ is supported by $x$ and $y$ in a small neighborhood of $S_r$ satisfying $\delta/2 \le d_0(x,y) \le \delta$.
	\end{Lemma}

	Lemma ~\ref{unkernel} follows from essentially the same proof as Lemma ~\ref{kernel}.
	Now we can follow the argument in Burq-G\'erard-Tzvetkov ~\cite{BGT} to finish the proof of Lemma ~\ref{Lunswitched}.
	
\subsection*{Argument from Burq-G\'erard-Tzvetkov \cite{BGT}}

	Let $T_\lambda$ be the operator with kernel
	\[
			\lambda^{1/2} e^{-i\lambda d_0(x,y)} A_\lambda(x,y)
		\]
	We will complete the proof of Lemma \ref{Lunswitched} by showing that for any smooth curve $\Gamma$ in $S_r$ of unit length,
	\begin{equation}
		\label{bgtgoal}
			\| T_\lambda f \|_{L^4(\Gamma)} \lesssim \lambda^{1/4} \| f \|_{L^2(\mathbb{R}^2)}
		\end{equation}
	By using a partition of unity and abusing notation, we can assume there is a point $x_0$ in $S_r$ such that
		$A_\lambda$ is supported by $x$ in the geodesic ball $B(x_0,c_0\delta)$ of radius $c_0\delta$ around $x_0$, where $c_0>0$ is small.
	Then there are small constants $c_2 > c_1 > 0$ such that $A_\lambda$ is supported by $y$ in the geodesic annulus $B(x_0, c_2 \delta) \smallsetminus B(x_0, c_1 \delta)$.
	
	Let $T$ be the tangent plane at $x_0$.
	We will use geodesic polar coordinates $(\rho, \omega)$ for the $y$-variable, with $\omega$ a unit vector in $T$ and $\rho >0$, so that $y=\exp_{x_0}(\rho\omega)$.
	Then we can write
	\[
			(T_\lambda f)(x) = \int_{c_1\delta}^{c_2\delta} (T_\lambda^\rho f_\rho)(x) \,d\rho
		\]
	with
	\[
			(T_\lambda^\rho f)(x) = \lambda^{1/2} \int_{S^1} e^{-i\lambda d_{0,\rho}(x,\omega)} A_{\lambda,\rho}(x,\omega) f(\omega) \,d\omega
		\]
	Here
	\[
			d_{0,\rho}(x,\omega)=d_0(x,y), \quad  f_\rho(\omega)=f(y), \quad \text{and} \quad A_{\lambda,\rho}(x,\omega)=J(\rho,\omega)A_\lambda(x,y)
		\]
	where $J$ is a smooth function satisfying $J(\rho,\omega)=\rho$ when $c_1\delta \le \rho \le c_2\delta$.
	
	If we can prove the uniform estimates
	\begin{equation}
		\label{Tr}
			\| T_\lambda^\rho f \|_{L^4(\Gamma)} \lesssim \lambda^{1/4} \| f \|_{L^2(S^1)}
		\end{equation}
	then \eqref{bgtgoal} will follow, because we will have
	\[
			\| T_\lambda f \|_{L^4(\Gamma)} \le \int_{c_1\delta}^{c_2\delta} \| T_\lambda^\rho f_\rho \|_{L^4(\Gamma)} \,d\rho
			\lesssim \lambda^{1/4} \int_{c_1\delta}^{c_2\delta} \| f_\rho \|_{L^2(S^1)} \,d\rho
			\lesssim \lambda^{1/4} \| f \|_{L^2(\mathbb{R})}
		\]
	So it suffices to prove \eqref{Tr}.
	By duality, \eqref{Tr} is equivalent to
	\begin{equation}
		\label{Trstar}
			\| (T_\lambda^\rho)^* f \|_{L^2(S^1)} \lesssim \lambda^{1/4} \| f \|_{L^{4/3}(\Gamma)}
		\end{equation}
	We will prove
	\begin{equation}
		\label{TrTrstar}
			\| T_\lambda^\rho(T_\lambda^\rho)^* f \|_{L^4(\Gamma)} \lesssim \lambda^{1/2} \| f \|_{L^{4/3}(\Gamma)}
		\end{equation}
	This will imply \eqref{Trstar}, because
	\[
			\| (T_\lambda^\rho)^* f \|_{L^2(S^1)}^2 = \int_\Gamma T_\lambda^\rho(T_\lambda^\rho)^* f(s) \overline{f(s)} \,ds
			\le \| T_\lambda^\rho(T_\lambda^\rho)^* f \|_{L^4(\Gamma)} \| f \|_{L^{4/3}(\Gamma)}
			\lesssim \lambda^{1/2} \| f \|_{L^{4/3}(\Gamma)}^2
		\]
	
	So it suffices to prove \eqref{TrTrstar}.
	Assume $x(t)$ parametrizes $\Gamma$ by arc length with domain $0 \le t \le 1$.
	The kernel of $T_\lambda^\rho (T_\lambda^\rho)^*$ is
	\[
			K_\lambda^\rho(t,\tau) =
				\lambda \int_{S^1} e^{-i\lambda [d_{0,\rho}(x(t), \omega) - d_{0,\rho}(x(\tau), \omega)]} A_{\lambda,\rho}(x(t),\omega) \overline{A_{\lambda,\rho}(x(\tau),\omega)} \,d\omega
		\]
	By making a linear change of variables, we may assume that $g_{ij}(x_0)=\delta^{ij}$.
	Then we have the following lemma, which we will use to control $K_\lambda^\rho$.
	
	\begin{Lemma}
		\label{bgtgauss}
			If $\rho>0$ is small and $\omega$ is in $S^1$, then
			\begin{equation}
				\label{bgtgaussconc}
					-\nabla_x d_{0,\rho}(x_0,\omega)=\omega
				\end{equation}
		\end{Lemma}
	
	\begin{proof}
	Let $\Theta$ be the geodesic sphere of radius $\rho$ around $y=\exp_{x_0}(\rho \omega)$.
	By Gauss' lemma, the vector $\omega$ is normal to $\Theta$ at $x_0$.
	Define a function $G$ by
	\[
		G(x)=d_{0,\rho}(x,\omega)
	\]
	Then $\Theta$ is a level set of $G$, so $\nabla G(x_0)$ is normal to $\Theta$ at $x_0$.
	That is, $\nabla G(x_0)$ is a multiple of $\omega$.
	Let $c$ be the geodesic satisfying $c(0)=x_0$ and $c'(0)=\omega$.
	Then for small $s$,
	\[
		G \big(c(s) \big)=\rho -s
	\]
	So $\nabla G(x_0) \cdot \omega =-1$.
	Since $\nabla G(x_0)$ is a multiple of $\omega$, this implies that $\nabla G(x_0)=-\omega$, which is \eqref{bgtgaussconc}.
	\end{proof}
	
	Using Lemma \ref{bgtgauss}, we can prove the following lemma.
	
	\begin{Lemma}
		\label{kerest}
			There is a $\delta_0>0$ such that if $|t-\tau| < \delta_0$, then
			\[
					|K_\lambda^\rho(t,\tau)| \lesssim \lambda (1+\lambda |t-\tau|)^{-1/2}
				\]
		\end{Lemma}
	
	\begin{proof}
	Define
	\[
			K_\lambda^\rho(x,x') = \lambda \int_{S^1} e^{-i\lambda [d_{0,\rho}(x, \omega) - d_{0,\rho}(x', \omega)]} A_{\lambda,\rho}(x,\omega) \overline{A_{\lambda,\rho}(x',\omega)} \,d\omega
		\]
	Since $\Gamma$ is smooth and parametrized by arc length, it suffices to show that
	\begin{equation}
		\label{kerxx}
			| K_\lambda^\rho(x,x') | \lesssim \lambda (1+\lambda |x-x'|)^{-1/2}
		\end{equation}
	We can write
	\[
			d_{0,\rho}(x,\omega)-d_{0,\rho}(x',\omega) = (x-x') \cdot \Psi_{0,\rho}(x,x',\omega)
		\]
	where
	\[
			\Psi_{0,\rho}(x,x',\omega) = \int_0^1 \nabla_x d_{0,\rho}\big(x'+s(x-x'),\omega\big) \,ds
		\]
	For $\sigma$ in $S^1$, define
	\[
			\Phi_{0,\rho}(x,x',\sigma,\omega) = \sigma \cdot \Psi_{0,\rho}(x,x',\omega)
		\]
	Now when $x \neq x'$,
	\[
			d_{0,\rho}(x,\omega)-d_{0,\rho}(x',\omega) = |x-x'| \Phi_{0,\rho}(x,x',\sigma_{x,x'},\omega)
		\]
	where
	\[
			\sigma_{x,x'}=\frac{x-x'}{|x-x'|}
		\]
	If we define
	\begin{equation}
		\label{JK}
			J_\mu^\rho(x,x',\sigma) = \int_{S^1} e^{-i\mu\Phi_{0,\rho}(x,x',\sigma,\omega)} A_{\lambda,\rho}(x,\omega) \overline{A_{\lambda,\rho}(x',\omega)} \,d\omega
		\end{equation}
	then it suffices to show that
	\begin{equation}
		\label{Jbounds}
			|J_\mu^\rho(x,x',\sigma)| \lesssim (1+\mu)^{-1/2}
		\end{equation}
		
	Parametrize $S^1$ by
	\[
			\omega(\theta)=(\cos\theta, \sin\theta)
		\]
	for $\theta$ in $[0, 2\pi )$.
	Write
	\[
			\sigma=(\cos\alpha, \sin\alpha)
		\]
	where $\alpha$ is in $[0,2\pi )$.
	Then by Lemma \ref{bgtgauss},
	\[
			\Phi_{0,\rho}(x_0,x_0,\sigma,\omega(\theta))=-\sigma \cdot \omega(\theta) = -\cos(\theta-\alpha)
		\]
	So we have
	\[
			\partial_\theta \Phi_{0,\rho}(x_0,x_0,\sigma,\omega(\theta)) = \sin(\theta-\alpha)
		\]
	and
	\[
			\partial_\theta^2 \Phi_{0,\rho}(x_0,x_0,\sigma,\omega(\theta)) = \cos(\theta-\alpha)
		\]
	There are relatively open sets $A$ and $B$, with $A \cup B = [0, 2\pi )$, such that for $\theta$ in $A$,
	\[
			|\partial_\theta \Phi_{0,\rho}(x_0,x_0,\sigma,\omega(\theta))| \ge c_A
		\]
	and for $\theta$ in $B$,
	\[
			|\partial_\theta^2 \Phi_{0,\rho}(x_0,x_0,\sigma,\omega(\theta))| \ge c_B
		\]
	Here $c_A$ and $c_B$ are positive constants.
	By continuity, if $\delta$ is sufficiently small and $x$, $x'$ are in $B(x_0, c_0\delta)$, then for $\theta$ in $A$,
	\begin{equation}
		\label{nonstph}
			|\partial_\theta \Phi_{0,\rho}(x,x',\sigma,\omega(\theta))| \ge c_A/2
		\end{equation}
	and for $\theta$ in $B$
	\begin{equation}
		\label{stph}
			|\partial_\theta^2 \Phi_{0,\rho}(x,x',\sigma,\omega(\theta))| \ge c_B/2
		\end{equation}
	By using a partition of unity on $S^1$ and abusing notation, it suffices to prove \eqref{Jbounds} in two cases.
	In the first case, we assume that \eqref{nonstph} holds on the support of the amplitude in \eqref{JK}.
	This case can be handled by integrating by parts, which yields much stronger bounds than in \eqref{Jbounds}.
	In the second case, we assume that \eqref{stph} holds on the support of the amplitude in \eqref{JK}.
	This case can be handled by using stationary phase, which yields \eqref{Jbounds}.
	\end{proof}
	
	Now we can use Lemma \ref{kerest} and the Hardy-Littlewood fractional integration inequality to obtain
	\[
			\| T_\lambda^\rho (T_\lambda^\rho)^* f \|_{L^4(\gamma)}
			\lesssim
			\Big\| \int_0^1 \lambda (1+\lambda|t-\tau|)^{-1/2} f\big(x(\tau)\big) \,d\tau \Big\|_{L^4(0,1)}
			\lesssim \lambda^{1/2} \| f \|_{L^{4/3}(\gamma)}
		\]
	This is \eqref{TrTrstar}, so we have proven Lemma \ref{Lunswitched}.
	Now Theorem \ref{thbgt} follows.

\section{Proof of Corollary 1.2}
	Fix $\delta>0$.
	Recall the set
	\begin{equation*}
		H_\delta = \Big\{x \in M : d(x, \partial M) \le \delta \Big\}
	\end{equation*}
	and recall that $E_\delta$ is the complement of $H_\delta$ in $M$.
	Also recall that we are assuming $M$ is a subset of a compact Riemannian manifold $(M_0,g)$ and that $\Delta_0$ is the Laplacian on $M_0$. 
	If $\delta > 0$ is small enough, then we can break up $\gamma$ into $\gamma \cap E_\delta$ and $\gamma \cap H_\delta$,
		where $\gamma \cap H_\delta$ is a broken geodesic with length at most $c_0 \delta^{1/2}$ for some fixed constant $c_0 > 0$.
	This is because the boundary is strictly geodesically concave.
	We can use H\"older's inequality and Theorem~\ref{thbgt} to control $\| e_j \|_{L^2(\gamma \cap H_\delta)}$. This gives
	\begin{equation}
		\label{gamma2}
			\| e_j \|_{L^2(\gamma \cap H_\delta)} \lesssim \delta^{1/8} \| e_j \|_{L^4(\gamma \cap H_\delta)} \lesssim \delta^{\frac{1}{8}} \lambda_j^{\frac{1}{4}}
		\end{equation}
	
	Choose $\chi \in \cal{S}(\R)$ with $\chi(0)=1$ and $\hat{\chi}$ supported on a closed interval contained strictly inside of $(\frac{1}{2}\delta,\delta)$.
	Define $\chi_\lambda(s)=\chi(s-\lambda)$ and $\rho_\lambda(s)=\chi_\lambda(s)+\chi_\lambda(-s)$.
	For large $\lambda$, we have
	\[
		1/2 \le | \rho_\lambda(\lambda) |
	\]
	To control $\| e_j \|_{L^2(\gamma \cap E_\delta)}$ we will use the following inequality.
	\begin{Theorem}
		\label{specbou}
			Let $p \ge 2$ and assume $\delta$ is small. If $\gamma$ is a unit length geodesic on $M_0$ and $\lambda \ge 1$, then there is a constant $C_\delta$ independent of the choice of $\gamma$ such that
			\[
					\| \rho_\lambda(\sqrt{-\Delta_0}) f \|_{L^2(\gamma)}+ \| \chi_\lambda(\sqrt{-\Delta_0}) f \|_{L^2(\gamma)} \le C_\delta \lambda^{\frac{1}{2p}} \| f \|_{L^p(M_0)}
				\]
		\end{Theorem}
	Bourgain \cite{B} proved this inequality for $\chi_\lambda$, and the inequality for $\rho_\lambda$ follows easily from Lemma ~\ref{neglap}.
	Recall \eqref{tycos}, which says
	\[
		(\rho_\lambda(\sqrt{-\Delta_g})f)\big{|}_{\gamma \cap E_\delta}=(\rho_\lambda(\sqrt{-\Delta_0})f)|_{\gamma \cap E_\delta}
	\]
	for $f$ in $L^2(M)$.
	So Theorem \ref{specbou} yields
	\begin{equation}
		\label{gamma1}
			\| e_j \|_{L^2(\gamma \cap E_\delta)} \le C_\delta \lambda_j^{\frac{1}{2p}} \| e_j \|_{L^p(M)}
		\end{equation}
	Now if $\delta$ is sufficiently small, Corollary~\ref{cob} follows from \eqref{gamma2} and \eqref{gamma1}.
	
\section{Proof of Proposition 1.6}
	
	For sufficiently small $\delta > 0$, we can break up $\gamma$ into $\gamma \cap E_\delta$ and $\gamma \cap H_\delta$,
		where $\gamma \cap H_\delta$ is a broken geodesic with length at most $c_0 \delta^{1/2}$ for some fixed constant $c_0 > 0$.
	This is because the boundary is strictly geodesically concave.
	By H\"older's inequality and Theorem \ref{thbgt},
	\[
			\limsup_{j \to \infty} \lambda_j^{-1/4} \| e_j \|_{L^2(\gamma \cap H_\delta)}
			\lesssim \limsup_{j \to \infty} \lambda_j^{-1/4} \delta^{1/8} \| e_j \|_{L^4(\gamma \cap H_\delta)}
			\lesssim \delta^{1/8}
		\]	
	Now it suffices to prove
	\[
			\limsup_{j \to \infty} \lambda_j^{-1/4} \| e_j \|_{L^2(\gamma \cap E_\delta)} = 0
		\]
	By breaking up $\gamma \cap E_\delta$ into pieces and abusing notation, we may assume that $\gamma$ is a geodesic in $M$ with $d_g(\gamma, \partial M) \ge \delta$ and moreover,
		that $\gamma$ is of length $L$ where $L$ is small and may depend on $\delta$.
	With these assumptions, we can follow the proof by Sogge \cite{Sog} for the boundaryless version of this problem, making only very minor modifications.
	 
	The proof will make use of Fermi normal coordinates about $\gamma$.
	These coordinates are well-defined on some neighborhood $W$ of $\gamma$.
	In this coordinate system, $\gamma$ becomes $\{(s,0): s \in [0,L] \}$ and the metric satisfies
	\[
		g_{ij}(s,0)=\delta^{ij}
	\]
	In the Fermi coordinates, the principal symbol $p_0$ of $\sqrt{-\Delta_0}$ satisfies
	\[
		p\big( (s,0),\xi \big)=|\xi|
	\]
	Let $\psi \in C_0^\infty(M)$ be supported strictly inside $W \cap E_{\delta/2}$ with $\psi=1$ on $\gamma$.
	Let $A$, $B_1$, and $B_2$ be pseudodifferential operators of order zero with symbols satisfying
	\[
			\psi(x)=A(x,\xi)+B_1(x,\xi)+B_2(x,\xi)
		\]
	In the Fermi coordinates, assume that $A$ is supported outside a conic neighborhood of the $\xi_1$-axis, $B_1$ is essentially supported in a conic neighborhood of the positive $\xi_1$-axis,
		and $B_2$ is essentially supported in a conic neighborhood of the negative $\xi_1$-axis.
	We also assume that $Af=0$ if $f$ is supported in $H_{\delta/2}$.
	
	Fix a positive integer $N$ and a real-valued $\chi \in \mathcal{S}(\mathbb{R})$ with $\chi(0)=1$.
	Assume the support of $\hat{\chi}$ is strictly inside $(-1/2,1/2)$.
	Define $\chi_{N,\lambda}(s)=\chi(N(s-\lambda))$ and $\rho_\lambda(s)=\chi(\delta(s-\lambda))+\chi_\lambda(\delta(-s-\lambda))$.
	Then
	\[
		\chi_{N,\lambda}(\lambda)=1
	\]
	and for large $j$,
	\[
		1/2 \le | \rho_\lambda(\lambda) |
	\]
	Let $B=B_1+B_2$.
	It suffices to show
	\[
			\| A \rho_\lambda(\sqrt{-\Delta_g})f \|_{L^2(\gamma)} + \| B \chi_{N,\lambda}(\sqrt{-\Delta_g})f \|_{L^2(\gamma)} \le C N^{-1/2} \lambda^{1/4} \|f\|_{L^2(M)}+C_N \|f\|_{L^2(M)}
		\]
	We have
	\[
		A \rho_\lambda(\sqrt{-\Delta_g})f=(\delta\pi)^{-1} \int \hat{\chi}(t/\delta)e^{-it\lambda} A \cos(t\sqrt{-\Delta_g})f \,dt
	\]
	Note the support of the integrand is strictly inside $(-\delta/2, \delta/2)$.
	
	The operator $U$ defined by $Uf(t,x)=\cos(t\sqrt{-\Delta_0})f(x)$ is a Fourier integral operator from $M_0$ to $M_0 \times \mathbb{R}$.
	Its canonical relation is
	\[
			\Big\{ (x,t,\xi, \tau; y, \eta): (x,\xi)=\Phi_t(y,\eta), \pm \tau=p_0(x,\xi) \Big\}
		\]
	where $\Phi_t: T^*M_0 \to T^*M_0$ is the geodesic flow on the cotangent bundle of $M_0$.
	The operator $V$ defined by $Vf(t,x)=\big{(} \cos(t\sqrt{-\Delta_0})f\big{)} \big{\vert}_\gamma(x)$ is a Fourier integral operator from $M_0$ to $\gamma \times \mathbb{R}$.
	Using the Fermi normal coordinates, we can write its canonical relation as
	\[
			\mathcal{C}=\Big\{ \big{(}(s,0),t,\xi_1,\tau;y,\eta\big{)}: \big{(}(s,0),\xi \big{)}=\Phi_t(y,\eta), \pm \tau = |\xi| \Big\}
		\]
	Then the projection from $\mathcal{C}$ to $T^*(\gamma \times \mathbb{R})$ is given by the map
	\[
		(s,t,\xi) \to (s,t,\xi_1, |\xi| )
	\]
	This has surjective differential away from $\xi_2=0$.

	If $|t| < \delta/2$, then by our assumptions on $A$,
	\[
			A \big{(} \cos(t\sqrt{-\Delta_g}) f\big{)} =A \big{(} \cos(t\sqrt{-\Delta_0})f\big{)}
		\]
	Define an operator by
	\[
			f \to \big{(}A ( \cos(t\sqrt{-\Delta_0})f)\big{)} \big{\vert}_\gamma
		\]
	This is a non-degenerate Fourier integral operator of order zero, because $A$ is supported away from the $\xi_1$-axis.
	This implies that
	\[
			\int | \hat{\chi}(t/\delta)| \, \| A \big{(} \cos(t\sqrt{-\Delta_g}) f\big{)} \|_{L^2(\gamma)} dt \lesssim \| f \|_{L^2(M)}
		\]
	which yields
	\[
			\| A \rho_\lambda(\sqrt{-\Delta_g})f \|_{L^2(\gamma)} \lesssim \| f \|_{L^2(M)}
		\]
		
	It remains to control the operators $\chi_\lambda^{N,B_j}$ defined by
	\[
			\chi_\lambda^{N,B_j} f = B_j \circ \chi(N(\sqrt{-\Delta_g}-\lambda))f=N^{-1}\int \hat{\chi}(t/N) e^{-it\lambda} \Big{(}B_j \circ e^{it\sqrt{-\Delta_g}} \Big{)} f dt
		\]
	Define an operator $V_j$ by
	\[
		V_j f(t,x) =  \Big( (B_j \circ e^{it\sqrt{-\Delta_g}} \circ B_j^*)f \Big) (x)
	\]
	Fix a distribution $u$ supported in the interior of $M$.
	Assume that $(t,x,\tau,\xi)$ is in the wave front set of $V_j u$.
	Then $(x,\xi)$ is in the essential support of $B_j$, and for some $(y,\eta)$ in the essential support of $B_j$,
		there is a broken geodesic $\Gamma$ satisfying $\Gamma(0)=y$, $\Gamma'(0)=\eta$, $\Gamma(t)=x$ and $\Gamma'(t)=\xi$.
	Since $\gamma$ is not contained in a periodic broken geodesic, the cutoffs $\psi$ and $B_j$ can be chosen with sufficiently small supports so that $V_j u$ is a smooth function over $2L \le |t| \le N+1$.
	That is, the operator $V_j$ is smoothing over the region $2L \le |t| \le N+1$.
	
	Define an operator $U_j$ by
	\[
		U_j f(t,x) =  \Big( (B_j \circ e^{it\sqrt{-\Delta_0}} \circ B_j^*)f \Big) (x)
	\]
	Then the operator $V_j-U_j$ is smoothing over the region $|t| \le 10L$, if $L$ is small.
		
	Let $T$ be the operator $f \to (\chi_\lambda^{N,B_j}f)\big{|}_\gamma$.
	We want to show that
	\[
			\| Tf\|_{L^2(M)} \le (CN^{-1/2} \lambda^{1/4}+C_{N,B_j}) \| f \|_{L^2(\gamma)}
		\]
	We will use the $TT^*$ method.
	We have
	\[
			\| T^*g \|_{L^2(M)}^2=\int_M T^*g \overline{T^*g} \,dx=\int_\gamma \big( TT^*g \big) \overline{g} \,ds \le \| TT^*g \|_{L^2(\gamma)} \| g \|_{L^2(\gamma)}
		\]
	So by duality, it suffices to prove that
	\begin{equation}
		\label{ttstar}
			\| TT^*g \|_{L^2(\gamma)} \le (CN^{-1} \lambda^{1/2}+C_{N,B_j}) \| g \|_{L^2(\gamma)}
		\end{equation}
		
	Let $w(\tau)=(\chi(\tau))^2$.
	Then the kernel of $TT^*$ is $K(\gamma(s),\gamma(s'))$ where $K(x,y)$ is the kernel of the operator $B_j \circ w(N(\sqrt{-\Delta_g}-\lambda)) \circ B_j^*$.
	Also $\hat{w}$ is supported in $[-1,1]$, since $\hat{w}=\hat{\chi}*\hat{\chi}$.
	Now
	\[
			B_j \circ w(N(\sqrt{-\Delta_g}-\lambda)) \circ B_j^* = N^{-1} \int \hat{w}(t/N) e^{-it\lambda} \Big{(}B_j \circ e^{it\sqrt{-\Delta_g}} \circ B_j^*\Big{)} \,dt
		\]
	Let $\varphi \in C_0^\infty(\mathbb{R})$ be supported on $[-1,1]$ with $\varphi = 1$ on $[-1/2,1/2]$.
	Now, by the smoothing properties of the operators $V_j$ and $V_j-U_j$, the difference between $B_j \circ w(N(\sqrt{-\Delta_g}-\lambda)) \circ B_j^*$ and
	\begin{equation}
		\label{lastpiece}
			N^{-1} \int \varphi (t/5L) \hat{w}(t/N)e^{-it\lambda} \Big{(}B_j \circ e^{it\sqrt{-\Delta_0}} \circ B_j^*\Big{)} \,dt
		\end{equation}
	has a kernel which is $\mathcal{O}(\lambda^{-m})$ for all $m$, so it remains to control the kernel of the operator \eqref{lastpiece}.
	If $5L$ is less than the injectivity radius of $M_0$, then the Hadamard parametrix can be used here.
	Then by stationary phase arguments, it follows that the kernel of the operator \eqref{lastpiece} satisfies
	\[
			|K(x,y)| \le CN^{-1}\lambda^{1/2}(d_g(x,y))^{-1/2}+C_{B_j}
	\]
	This yields \eqref{ttstar}, completing the proof of Proposition \ref{open}.

\linespread{2}


\begin{thebibliography}{9}

\bibitem{B}Bourgain, J. (2009). Geodesic restrictions and $L^p$-estimates for eigenfunctions of Riemannian surfaces. \emph{Linear and Complex Analysis: Dedicated to V. P. Havin on the Occasion of His 75th Birthday, American Math. Soc. Translations, Advances in the Mathematical Sciences}, 27-35.

\bibitem{BGT}Burq, N., G\'erard, P. and Tzvetkov, N. (2007). Restriction of the Laplace-Beltrami eigenfunctions to submanifolds. \emph{Duke Math. J.} \textbf{138}, 445-486.

\bibitem{G}Grieser, D. (1992). \emph{$L^p$ bounds for eigenfunctions and spectral projections of the Laplacian near concave boundaries}. Ph.D. Thesis. University of California, Los Angeles: USA.

\bibitem{H}H\"ormander, L. (1971). Fourier integral operators. I. \emph{Acta Math.} \textbf{127}, 79-183.

\bibitem{MTBook}Melrose, R. and Taylor, M. Boundary problems for the wave equation with grazing and gliding rays. Manuscript.

\bibitem{R}Reznikov, A. Norms of geodesic restrictions for eigenfunctions on hyperbolic surfaces and representation theory. Preprint, arXiv:math.AP/0403437.

\bibitem{SSog2}Smith, H. and Sogge, C.D. (1994). $L^p$ regularity for the wave equation with strictly convex obstacles. \emph{Duke Math. J.} \textbf{73}, 97-155.

\bibitem{SSog}Smith, H. and Sogge, C.D. (1995). On the critical semilinear wave equation outside convex obstacles. \emph{J. Amer. Math. Soc.} \textbf{8}, 879-916.

\bibitem{SogBook}Sogge, C.D. (1993). \emph{Fourier integrals in classical analysis}. Cambridge: Cambridge University Press.

\bibitem{Sog}Sogge, C.D. Kakeya-Nikodym averages and $L^p$-norms of eigenfunctions. \emph{To appear in Tohoku Math. J.}

\bibitem{ZZ}Zelditch, S. and Zworski, M. (1996). Ergodicity of eigenfunctions for ergodic billiards. \emph{Comm. Math. Phys.} \textbf{175}, 673-682.

\bibitem{Z}Zworski, M. (1990). High frequency scattering by a convex obstacle. \emph{Duke Math. J.} \textbf{61}, 545-634.

\end{thebibliography}
\end{document}